\documentclass[reqno]{amsart}

\usepackage{pgfplots}
\usepackage{amsthm}
\usepackage{latexsym}
\usepackage{dsfont}
\usepackage{bbm}
\usepackage{amssymb}
\usepackage{amsmath}
\usepackage{mathtools}
\usepackage{tikz}

\usepackage{enumerate}

\newcommand{\defeq}{\vcentcolon=}
\newcommand{\eqdef}{=\vcentcolon}
\newcommand{\Law}{\mathcal{L}}

\newcommand{\weaklyto}{%
  \mathrel{\vbox{\offinterlineskip\ialign{%
    \hfil##\hfil\cr
    $\scriptscriptstyle\mathrm{w}$\cr
    $\to$\cr
}}}}
\newcommand{\distto}{%
  \mathrel{\vbox{\offinterlineskip\ialign{%
    \hfil##\hfil\cr
    $\scriptscriptstyle\mathrm{d}$\cr
    \noalign{\kern-.05ex}
    $\to$\cr
}}}}
\newcommand{\Probto}{%
  \mathrel{\vbox{\offinterlineskip\ialign{%
    \hfil##\hfil\cr
    $\scriptscriptstyle\Prob$\cr
    \noalign{\kern-.05ex}
    $\to$\cr
}}}}
\newcommand{\TVto}{%
  \mathrel{\vbox{\offinterlineskip\ialign{%
    \hfil##\hfil\cr
    $\scriptscriptstyle\mathrm{TV}$\cr
    \noalign{\kern-.05ex}
    $\to$\cr
}}}}

\newcommand{\N}{\mathbb{N}}
\newcommand{\Z}{\mathbb{Z}}
\newcommand{\Q}{\mathbb{Q}}
\newcommand{\R}{\mathbb{R}}

\newcommand{\C}{\mathbb{C}}
\newcommand{\Chat}{\hat{\mathbb{C}}}
\newcommand{\Gen}{\mathcal{G}}
\newcommand{\Dom}{\mathcal{D}}
\newcommand{\I}{\mathcal{I}}
\newcommand{\V}{\mathbb{V}}

\newcommand{\Rp}{\R_{>}}
\newcommand{\imag}{\mathrm{i}}
\newcommand{\Real}{\mathrm{Re}}
\newcommand{\Imag}{\mathrm{Im}}
\newcommand{\Surv}{\mathcal{S}}

\newcommand{\comp}{\mathsf{c}}
\newcommand{\transp}{\mathsf{T}}

\newcommand{\Prob}{\mathbb{P}}
\DeclareMathOperator{\Var}{\mathbb{V}\mathrm{ar}}
\DeclareMathOperator{\Cov}{\mathrm{Cov}}

\DeclareMathOperator{\inner}{\mathrm{int}}

\newcommand{\E}{\mathbb{E}}
\newcommand{\A}{\mathcal{A}}

\newcommand{\F}{\mathcal{F}}
\renewcommand{\H}{\mathcal{H}}
\newcommand{\cZ}{\mathcal{Z}}
\newcommand{\1}{\mathbbm{1}}

\newcommand{\eqdist}{%
  \mathrel{\vbox{\offinterlineskip\ialign{%
    \hfil##\hfil\cr
    $\scriptscriptstyle\mathrm{law}$\cr
    \noalign{\kern.2ex}
    $=$\cr
}}}}

\newcommand{\supp}{\mathrm{supp}}
\newcommand{\gr}{\mathbb{U}}

\newcommand{\ds}{\mathrm{d} \mathit{s}}
\newcommand{\dt}{\mathrm{d} \mathit{t}}

\newcommand{\dx}{\mathrm{d} \mathit{x}}
\newcommand{\dy}{\mathrm{d} \mathit{y}}
\newcommand{\dz}{\mathrm{d} \mathit{z}}

\newcommand{\dmu}{\mathrm{d} \mu}
\newcommand{\dnu}{\mathrm{d} \nu}

\theoremstyle{plain}
\newtheorem{theorem}{Theorem}[section]
\newtheorem{corollary}[theorem]{Corollary}
\newtheorem{lemma}[theorem]{Lemma}
\newtheorem{proposition}[theorem]{Proposition}
\theoremstyle{remark}

\theoremstyle{definition}

\newtheorem{example}[theorem]{Example}

\numberwithin{equation}{section}

\begin{document}

%
%

\title{Fluctuations of Biggins' martingales at complex parameters}

\authors{%
Alexander~Iksanov\footnote{
Taras Shevchenko National University of Kyiv, Ukraine. Email: iksan@univ.kiev.ua}
\and
	Konrad~Kolesko\footnote{Universit\"at Innsbruck, Austria
	and
	Uniwersytet Wroc\l{}awski, Wroc\l{}aw, Poland.
	Email: kolesko@math.uni.wroc.pl}
  \and	Matthias~Meiners\footnote{Universit\"at Innsbruck, Austria.
  Email: matthias.meiners@uibk.ac.at}}

\begin{abstract}
The long-term behavior of a supercritical branching random walk
can be described and analyzed with the help of Biggins'
martingales, parame\-trized by real or complex numbers. The study
of these martingales with complex parameters is a rather recent
topic. Assuming that certain sufficient conditions for the
convergence of the martingales to non-degenerate limits hold, we
investigate the fluctuations of the martingales around their
limits. We discover three different regimes. First, we show that
for parameters with small absolute values, the fluctuations are
Gaussian and the limit laws are scale mixtures of the real or
complex standard normal laws.
We also cover the boundary of this
phase. Second, we find a region in the parameter space in which
the martingale fluctuations are determined by the extremal
positions in the branching random walk. Finally, there is a
critical region (typically on the boundary of the set of
parameters for which the martingales converge to a non-degenerate
limit) where the fluctuations are stable-like and the limit laws
are the laws of randomly stopped L\'evy processes satisfying
invariance properties similar to stability.
\smallskip

\\{\bf Subclass:} MSC 60J80 $\cdot$ MSC 60F15
\end{abstract}

\maketitle

\section{Introduction}	\label{sec:intro}

We consider a discrete-time supercritical branching random walk on the real line.
The distribution of the branching random walk is governed by a point process $\cZ$ on $\R$.
Although there are numerous papers in which $\cZ(\R)$ is allowed to be infinite with
positive probability, the standing assumption of the present paper is $\cZ(\R)<\infty$ almost surely (a.\,s.).
At time $0$, the process starts with one individual (also called particle), the ancestor, which resides at the origin.
At time $1$, the ancestor dies and simultaneously places offspring on the real line with positions given by the points of the point process $\cZ$.
The offspring of the ancestor form the first generation of the branching random walk.
At time $2$, each particle of the first generation dies and has offspring with positions relative to their parent's position given by an independent copy of $\cZ$.
The individuals produced by the first generation particles form the second generation of the process,
and so on.

The sequence of (random) Laplace transforms of the point process of the $n$th generation positions, evaluated at an appropriate $\lambda \in \C$ and suitably normalized, forms a martingale.
This martingale that we denote by $(Z_n(\lambda))_{n\in\N_0}$, where $\N_0\defeq \N \cup \{0\}$ and $\N\defeq \{1,2,\ldots\}$, is called additive martingale or Biggins' martingale.
These martingales
play a key role in the study of the branching random walk, see e.g.\ \cite[Theorem 4]{Biggins:1992}
where the spread of the $n$th generation particles
is described in terms of additive martingales.
In the same paper, Biggins showed that subject to some mild conditions,
$Z_n(\lambda)$ converges almost surely to some limit $Z(\lambda)$ locally uniformly in $\lambda$ from a certain open domain $\Lambda \subseteq \C$.
Biggins' result was extended by two of the three present authors \cite{Kolesko+Meiners:2017} to (parts of) the boundary of the set $\Lambda$.
It is natural to ask for the rate of this convergence, i.e., for the fluctuations of $Z(\lambda)-Z_n(\lambda)$.

A partial answer to this question was given by Iksanov and Kabluchko \cite{Iksanov+Kabluchko:2016},
who proved a functional central limit theorem with a random centering for Biggins' martingale for real and sufficiently small $\lambda$.
The counterpart of this result in the context of the complex branching Brownian motion energy model has been derived by Hartung and Klimovsky \cite{Hartung+Klimovsky:2017}.
Another related statement for branching Brownian motion can be found in the recent paper by Maillard and Pain \cite{Maillard+Pain:2018},
where the fluctuations of the derivative martingale for branching Brownian motion are studied. The authors of the present paper also
investigated in \cite{Iksanov+al:2018} fluctuations of $Z(\lambda)-Z_n(\lambda)$ for real $\lambda$ in the regime
where the distribution of $Z_1(\lambda)$ belongs to the normal domain of attraction of an $\alpha$--stable law with $\alpha \in (1,2)$.
We refer to the end of Section \ref{sec:model and results} for a detailed account of the existing literature.

The aim of the paper at hand is to give a complete description of the fluctuations of Biggins' martingales whenever they converge
while making only minimal moment assumptions.
It turns out that, apart from the \emph{Gaussian regime} studied in \cite{Iksanov+Kabluchko:2016}, there are two further cases.
There is an \emph{extremal regime}, where the fluctuations are determined by the particles close to the minimal position in the branching random walk.
In this regime, the fluctuations are exponentially small with a polynomial correction.
And finally, there is a \emph{critical stable regime} with fluctuations of polynomial order.

\section{Model description and main results}	\label{sec:model and results}

We continue with the formal definition of the branching random walk and a review of the results
on which our work is based.

\subsection{Model description and known results}

\subsubsection*{The model.}
Set $\I \defeq  \bigcup_{n \geq 0} \N^n$. 
We use the standard Ulam-Harris notation, that is, for $u=(u_1,\ldots,u_n) \in \N^n$, we also write $u_1 \ldots u_n$.
Further, if $v=(v_1,\ldots,v_m) \in \N^m$, we write $uv$ for $(u_1,\ldots,u_n,v_1,\ldots,v_m)$.
For $k \leq n$, denote $u_1 \ldots u_{k}$, the ancestor of $u$ in generation $k$, by $u|_k$.
The ancestor of the whole population is identified with the empty tuple $\varnothing$ and its position is $S(\varnothing) = 0$.
Let $(\cZ(u))_{u \in \I}$ be a family of i.i.d.\ copies of the basic reproduction point process $\cZ$ defined on some probability space $(\Omega,\A,\Prob)$.
We write $\cZ(u) = \sum_{j=1}^{N(u)} X_j(u)$, where $N(u) = \cZ(u)(\R)$, $u \in \I$.
We assume that $\cZ(\varnothing) = \cZ$. In general, we drop the argument $\varnothing$ for quantities
derived from $\cZ(\varnothing)$, for instance, $N = N(\varnothing)$.
Generation $0$ of the population is given by $\Gen_0 \defeq \{\varnothing\}$
and, recursively,
\begin{equation*}
\Gen_{n+1} \defeq \{uj \in \N^{n+1}: u \in \Gen_n \text{ and } 1 \leq j \leq N(u)\}
\end{equation*}
is generation $n+1$ of the process. 
Define the set of all individuals by $\Gen \defeq \bigcup_{n \in \N_0} \Gen_n$.
The position of an individual $u = u_1 \ldots u_n \in \Gen_n$ is
\begin{equation*}
S(u)	\defeq	X_{u_1}(\varnothing) + \ldots + X_{u_n}(u_1 \ldots u_{n-1}).
\end{equation*}
The point process of the positions of the $n$th generation individuals will be denoted by $\cZ_n$, that is,
\begin{equation*}
\cZ_n = \sum_{|u|=n} \delta_{S(u)}
\end{equation*}
where here and in what follows, we write $|u|=n$ for $u \in \Gen_n$.
The sequence of point processes $(\cZ_n)_{n \in \N_0}$ is called a {\it branching random walk}.

We assume that $(\cZ_n)_{n \in \N_0}$ is supercritical, i.\,e., $\E[N]=\E[\cZ(\R)] > 1$.
Then the generation sizes $\cZ_n(\R)$, $n \in \N_0$ form a supercritical Galton-Watson process
and thus $\Prob(\Surv) > 0$ for the survival set
\begin{equation*}
\Surv \defeq \{\# \Gen_n > 0 \text{ for all } n \in \N\} = \{\cZ_n(\R) > 0 \text{ for all } n \in \N\}.
\end{equation*}
The Laplace transform of the intensity measure $\mu$ of $\cZ$ is the function
\begin{equation}	\label{eq:m}
\lambda ~\mapsto~ m(\lambda) \defeq \int_\R e^{-\lambda x} \, \mu(\dx) = \E\bigg[\sum_{|u|=1} e^{-\lambda S(u)} \bigg],	\qquad	\lambda \in \C
\end{equation}
where $\lambda = \theta + \imag \eta$ with $\theta,\eta \in \R$.
(We adopt the convention from \cite{Biggins:1992} and always write $\theta$ for $\Real(\lambda)$ and $\eta$ for $\Imag(\lambda)$.)
Throughout the paper, we assume that
\begin{equation*}
\Dom = \{\lambda \in \C: m(\lambda) \text{ converges absolutely}\} = \{\theta \in \R: m(\theta) < \infty\} + \imag \R	\text{ \ is non-empty.}
\end{equation*}
For $\lambda \in \Dom$ and $n \in \N_0$, let
\begin{equation*}
Z_n(\lambda)	\defeq	\frac{1}{m(\lambda)^n} \int_\R e^{-\lambda x} \, \cZ_n(\dx) = \frac{1}{m(\lambda)^n} \sum_{|u|=n} e^{-\lambda S(u)}.
\end{equation*}
Denote by $\F_n \defeq \sigma(\cZ(u): u \in \bigcup_{k=0}^{n-1} \N^k)$, and let $\F_{\infty} \defeq \sigma(\F_n: n \in \N_0)$.
It is well known and easy to check that $(Z_n(\lambda))_{n \in \N_0}$ forms a complex-valued martingale with respect to $(\F_n)_{n \in \N_0}$.
It is called {\it additive martingale in the branching random walk}
and also {\it Biggins' martingale} in honor of Biggins' seminal contribution \cite{Biggins:1977}.

\subsubsection*{Convergence of complex martingales.}

Convergence of these martingales has been investigated by various authors in the case $\lambda = \theta \in \R$,
see e.\,g.\ \cite{Alsmeyer+Iksanov:2009,Biggins:1977,Lyons:1997}.
For the complex case,
the most important sources for us are \cite{Biggins:1992} and \cite{Kolesko+Meiners:2017}.
Theorem 1 of \cite{Biggins:1992} states that if
\begin{equation*}	\tag{B1}	\label{eq:gamma moment Z_1(theta)}	\textstyle
\E[Z_1(\theta)^\gamma] < \infty	\quad	\text{for some } \gamma \in (1,2]
\end{equation*}
and
\begin{equation*}	\tag{B2}	\label{eq:contraction condition}	\textstyle
\frac{m(p\theta)}{|m(\lambda)|^p} < 1	\quad	\text{for some } p \in (1,\gamma],
\end{equation*}
then $(Z_n(\lambda))_{n \in \N_0}$ converges a.\,s.\ and in $L^p$ to a limit variable $Z(\lambda)$.
Theorem 2 in the same source gives that this convergence is locally uniform (a.\,s.\ and in mean)
on the set $\Lambda = \bigcup_{\gamma \in (1,2]} \Lambda_{\gamma}$
where $\Lambda_\gamma = \Lambda_{\gamma}^1 \cap \Lambda_{\gamma}^3$
and, for $\gamma \in (1,2]$,
\begin{equation*}	\textstyle
\Lambda_{\gamma}^1 =
\inner\{\lambda \in \Dom: \E[Z_1(\theta)^\gamma] < \infty\}
\text{ and }
\textstyle
\Lambda_{\gamma}^3 =
\inner\big\{\lambda \in \Dom: \inf_{1 \leq p \leq \gamma} \frac{m(p \theta)}{|m(\lambda)|^p} < 1\big\}.
\end{equation*}
In \cite{Kolesko+Meiners:2017}, convergence of the martingales $(Z_n(\lambda))_{n \in \N_0}$
for parameters $\lambda$ from the boundary $\partial \Lambda$ is investigated.
Theorem 2.1 in the cited article states that subject to the conditions
\begin{equation*}	\tag{C1}	\label{eq:characteristic index}	\textstyle
\frac{m(\alpha \theta)}{|m(\lambda)|^\alpha} = 1
\ \text{ and }\
\E\big[\sum_{|u|=1} \theta S(u) \frac{e^{-\alpha \theta S(u)}}{|m(\lambda)|^\alpha}\big] \geq -\log (|m(\lambda)|)
\text{ for some } \alpha \in (1,2)
\end{equation*}
and
\begin{equation*}	\tag{C2}	\label{eq:logarithmic moment condition}
\E[|Z_1(\lambda)|^\alpha \log_+^{2+\epsilon} (|Z_1(\lambda)|)] < \infty	\quad	\text{ for some } \epsilon > 0
\end{equation*}
with the same $\alpha$ as in \eqref{eq:characteristic index},
there is convergence of $Z_n(\lambda)$ to some limit variable $Z(\lambda)$.
The convergence holds a.\,s.\ and in $L^p$ for any $p < \alpha$.

As has already been mentioned the fluctuations of $Z_n(\lambda)$ around $Z(\lambda)$ as $n \to \infty$ are the subject of the present paper.
More precisely, we find (complex) scaling constants $a_n = a_n(\lambda) \not= 0$
such that $a_n (Z(\lambda)-Z_n(\lambda))$ converges in
distribution to a non-degenerate limit as $n \to \infty$.

\subsection{Main results}	\label{subsec:main results}

We give an example before we state our main results.

\subsubsection*{Example.}
There are three fundamentally different regimes for the fluctuations of $Z_n(\lambda)$ around its limit $Z(\lambda)$.
These regimes are best understood via an example, which is in close analogy to branching Brownian motion.

\begin{example}[Binary splitting and Gaussian increments]	\label{Exa:binary Gaussian}
Consider a branching random walk with binary splitting and independent standard Gaussian increments,
that is, $\cZ = \delta_{X_1}+\delta_{X_2}$
where $X_1,X_2$ are independent standard normals.
Then $m(\lambda) = 2 \exp(\lambda^2/2)$ for $\lambda \in \C$.
For every $\theta \in \R$ and $\gamma > 1$, we have
$\E[Z_1(\theta)^\gamma]<\infty$.
Hence
$\Lambda = \{\lambda \in \C: m(p\theta)/|m(\lambda)|^p < 1 \text{ for some } p \in (1,2]\}$.
Thus, $\lambda \in \Lambda$ if and only if there exists some $p \in (1,2]$ with $m(p\theta)/|m(\lambda)|^p < 1$.
The latter inequality is equivalent to
\begin{equation}	\label{eq:binary Gaussian condition}	\textstyle
(1-p) 2 \log 2 + p^2 \theta^2 - p(\theta^2-\eta^2) < 0.
\end{equation}
It follows from the discussion in \cite[Example 3.1]{Kolesko+Meiners:2017}
that $\lambda \in \Lambda$ iff
$|\theta| \leq \sqrt{2 \log 2}/2$ and $\theta^2 + \eta^2 < \log 2$,
or $\sqrt{2 \log 2}/2 \leq |\theta| < \sqrt{2 \log 2}$ and $|\eta|<\sqrt{2 \log 2}-|\theta|$.
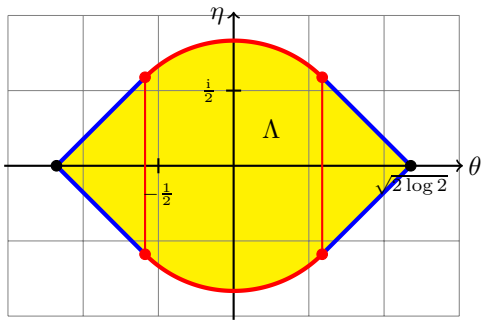
\begin{figure}[h]
\begin{center}
\begin{tikzpicture}[scale=1]
		\fill[yellow] plot[domain=-2*sqrt(2*ln(2)):-2*sqrt(ln(2)/2)] (\x, {2*sqrt(2*ln(2))+\x)})
					-- plot[domain=-2*sqrt(ln(2)/2):-2*sqrt(2*ln(2))] (\x, {-2*sqrt(2*ln(2))-\x)});
		\fill[yellow] plot[domain=2*sqrt(ln(2))/sqrt(2):2*sqrt(2*ln(2))] (\x, {2*sqrt(2*ln(2))-\x)})
					-- plot[domain=2*sqrt(2*ln(2)):2*sqrt(ln(2))/sqrt(2)] (\x, {-2*sqrt(2*ln(2))+\x)});
		\fill[yellow] plot[domain=-2*sqrt(ln(2)/2):2*sqrt(ln(2)/2)] (\x, {2*sqrt(ln(2)-(\x/2)*(\x/2))})
					-- plot[domain=2*sqrt(ln(2)/2):-2*sqrt(ln(2)/2)] (\x, {-2*sqrt(ln(2)-(\x/2)*(\x/2))});
		\draw [help lines] (-3,-2) grid (3,2);
		\draw [thick,->] (0,-2.05) -- (0,2.05);
		\draw [thick,->] (-3.05,0) -- (3.05,0);
		\draw (3,0) node[right]{$\theta$};
		\draw (0,2) node[left]{$\eta$};
		\draw (0.5,0.25) node[above]{$\Lambda$};
		\draw [ultra thick, domain=-2*sqrt(ln(2)/2):2*sqrt(ln(2)/2), samples=128, color=red] plot(\x, {2*sqrt(ln(2)-(\x/2)*(\x/2))});
		\draw [ultra thick, domain=-2*sqrt(ln(2)/2):2*sqrt(ln(2)/2), samples=128, color=red] plot(\x, {-2*sqrt(ln(2)-(\x/2)*(\x/2))});
		\draw [ultra thick, domain=2*sqrt(ln(2))/sqrt(2):2*sqrt(2*ln(2)), samples=128, color=blue] plot(\x, {2*sqrt(2*ln(2))-\x)});
		\draw [ultra thick, domain=2*sqrt(ln(2))/sqrt(2):2*sqrt(2*ln(2)), samples=128, color=blue] plot(\x, {-2*sqrt(2*ln(2))+\x)});
		\draw [ultra thick, domain=-2*sqrt(2*ln(2)):-2*sqrt(ln(2))/sqrt(2), samples=128,color=blue] plot(\x, {2*sqrt(2*ln(2))+\x)});
		\draw [ultra thick, domain=-2*sqrt(2*ln(2)):-2*sqrt(ln(2))/sqrt(2), samples=128, color=blue] plot(\x, {-2*sqrt(2*ln(2))-\x)});
		\draw[thick, color=red] (1.17741,1.17741) --(1.17741,-1.17741);
		\draw[thick, color=red] (-1.17741,1.17741) --(-1.17741,-1.17741);
		\filldraw[color=black] (2.35482,0) circle (2pt)	node[below]{\scriptsize $\sqrt{2 \log 2}$};
		\filldraw[color=black] (-2.35482,0) circle (2pt);
		\filldraw[color=red] (1.17741,1.17741) circle (2pt);
		\filldraw[color=red] (1.17741,-1.17741) circle (2pt);
		\filldraw[color=red] (-1.17741,1.17741) circle (2pt);
		\filldraw[color=red] (-1.17741,-1.17741) circle (2pt);
		\draw[thick] (-1,0.1) --(-1,-0.1) node[below]{\scriptsize $-\frac12$};
		\draw[thick] (0.1,1) -- (-0.1,1) node[left]{\scriptsize $\frac\imag2$};
\end{tikzpicture}
\caption{\small The figure shows the different regimes of fluctuations of Biggins' martingales
in the branching random walk with binary splitting and independent standard Gaussian increments.}
\label{fig:Gaussian Lambda}
\end{center}
\end{figure}
Corollary 1.1 of \cite{Iksanov+Kabluchko:2016} applies to parameters $\theta \in \R$ satisfying $m(2\theta)/m(\theta)^2 < 1$ or, equivalently, $|\theta|<\sqrt{\log 2}$.
In this case, the corollary gives convergence in distribution of $(\sqrt{2} \exp(-\theta^2/2))^{n}(Z(\theta)-Z_n(\theta))$
to a constant multiple of $\sqrt{Z(2\theta)} \cdot X$ where $X$ is real standard normal and independent of $Z(2\theta)$.
According to \cite{Lyons:1997}, $Z(2\theta)$ is non-degenerate iff $|\theta| < \sqrt{2 \log 2}/2$,
i.e., the limit in Corollary 1.1 of \cite{Iksanov+Kabluchko:2016} is non-degenerate only
for $\theta$ which are situated on the real axis strictly between the two red vertical lines in Figure \ref{fig:Gaussian Lambda}.
Our first result, Theorem \ref{Thm:Cor 1.1 in Iksanov+Kabluchko:2016} below, extends Corollary 1.1 from \cite{Iksanov+Kabluchko:2016}
and, in this particular example, gives the convergence of $(\sqrt{2} \exp((\lambda^2/2 - \theta^2))^{n} (Z(\lambda)-Z_n(\lambda))$
to a constant multiple of $\sqrt{Z(2 \theta)} \cdot X$ in the whole bounded yellow domain surrounded by red arcs and lines.
Here, $X$ is independent of $Z(2\theta)$ and complex standard normal if $\Imag(\lambda) \not = 0$.
For parameters $\lambda$ from the red vertical lines, the same limit relation holds, but the limit is degenerate as $Z(2\theta)=0$ a.\,s.
Indeed, $2\theta$ is then one of the two black dots in the figure.
This problem can be resolved with the help of Seneta-Heyde norming.
From \cite{Aidekon+Shi:2014} we know that $\sqrt{n} Z_n(2\theta)$
converges in probability to a constant multiple of the (non-degenerate) limit $D_\infty$ of the derivative martingale.
Modifying the scaling in Theorem \ref{Thm:Cor 1.1 in Iksanov+Kabluchko:2016} by the additional prefactor $n^{1/4}$
gives a nontrivial limit theorem where the limit is a constant multiple of $\sqrt{D_\infty} \cdot X$
with the same $X$ as before which is further independent of $D_\infty$.
This is the content of Theorem \ref{Thm:Gaussian boundary case} which applies to the $\lambda$ from the vertical red lines.

A similar trick does not work for parameters from the open yellow domains surrounded by the two triangles consisting of red vertical lines and diagonal blue lines.
There, the contribution of the minimal positions in the branching random walk to $Z(\lambda)-Z(\lambda)$ is too large for a limit theorem with a (randomly scaled)
normal or stable limit.
Instead, it can be checked that our Theorem \ref{Thm:domination of extremal particles} applies.
The most tedious part here is to show that $\E[|Z(\lambda)|^p]<\infty$ for some suitable $p$,
but this can be achieved by checking that the sufficient conditions \eqref{eq:gamma moment Z_1(theta)} and \eqref{eq:contraction condition} are fulfilled.
Theorem \ref{Thm:domination of extremal particles} is based on the convergence of the point process of the branching random walk seen from its tip \cite{Madaule:2017}.
The correct scaling factors provided by the theorem are $n^{3\lambda/(2\vartheta)} \cdot (2 \exp(\lambda^2/2)/(4^{\lambda/\vartheta}))^n$ with $\vartheta = \sqrt{2 \log 2}$
and the limit distribution has a random series representation involving the limit process of branching random walk seen from its tip.

Finally, on the blue lines, it holds that the distribution of the martingale limit $Z(\lambda)$ is in the domain of attraction of a stable law
and hence $Z(\lambda)-Z_n(\lambda)$ exhibits stable-like fluctuations.
This regime is covered by Theorem \ref{Thm:fluctuations on partial Lambda^(1,2)}, which shows that
$n^{\frac\lambda{2\alpha\theta}}(Z(\lambda)-Z_n(\lambda))$ converges in distribution to a L\'evy process independent of $\F_\infty$
satisfying an invariance property similar to $\alpha$-stability (the details are explained in Example \ref{Exa:binary Gaussian 2})
evaluated at the limit $D_\infty$ of the derivative martingale,
where $\alpha = \sqrt{2 \log 2}/\theta \in (1,2)$ for $\theta \in (\frac{1}{2} \sqrt{2 \log 2},\sqrt{2\log 2})$.
\end{example}

\subsubsection*{Weak convergence almost surely and in probability.}
If $\zeta,\zeta_1,\zeta_2,\ldots$ are random variables taking values in $\C$,
we write
\begin{equation}	\label{eq:convergence weakly in probability}
\Law(\zeta_n | \F_n) \weaklyto \Law(\zeta | \F_{\infty})	\quad	\text{in $\Prob$-probability}
\end{equation}
(in words, the distribution of $\zeta_n$ given $\F_n$ converges weakly to the distribution of $\zeta$ given $\F_\infty$ in $\Prob$-probability)
if for every bounded continuous function $\phi:\C \to \R$ it holds that
$\E[\phi(\zeta_n)|\F_n]$ converges to $\E[\phi(\zeta)|\F_\infty]$ in $\Prob$-probability as $n \to \infty$.
Notice that \eqref{eq:convergence weakly in probability} implies that $\zeta_n$ converges to $\zeta$ in distribution as $n \to \infty$
as for any bounded and continuous function $\phi:\C \to \R$ and every strictly increasing sequence of positive integers,
we can extract a subsequence $(n_k)_{k \in \N}$ such that $\E[\phi(\zeta_{n_k})|\F_{n_k}]$ converges to $\E[\phi(\zeta)|\F_\infty]$ a.\,s.
Hence, by the dominated convergence theorem,
\begin{equation*}
\E[\phi(\zeta_{n_k})] = \E[\E[\phi(\zeta_{n_k})|\F_{n_k}]] \to \E[\E[\phi(\zeta)|\F_\infty]] = \E[\phi(\zeta)]	\quad	\text{as } k \to \infty.
\end{equation*}
This implies $\E[\phi(\zeta_n)] \to \E[\phi(\zeta)]$ as $n \to \infty$ and, therefore, $\zeta_n \distto \zeta$.

Analogously, we write
\begin{equation}	\label{eq:convergence weakly a.s.}
\Law(\zeta_n | \F_n) \weaklyto \Law(\zeta | \F_{\infty})	\quad	\text{$\Prob$-a.\,s.}
\end{equation}
(in words, the distribution of $\zeta_n$ given $\F_n$ converges a.\,s.\ to the distribution of $\zeta$ given $\F_\infty$)
if for every bounded continuous function $\phi:\C \to \R$ it holds that
$\E[\phi(\zeta_n)|\F_n]$ converges to $\E[\phi(\zeta)|\F_\infty]$ a.\,s.\ as $n \to \infty$.
Clearly, also $\Law(\zeta_n | \F_n) \weaklyto \Law(\zeta | \F_{\infty})$ $\Prob$-a.\,s.\ implies $\zeta_n \distto \zeta$.
\smallskip

Henceforth, we shall assume that $\lambda \in \Dom$ satisfies $\theta \geq 0$.
This simplifies the presentation of our results but is not a restriction of generality. Indeed, if $\theta < 0$,
we may replace the point process $\cZ = \sum_{j=1}^N \delta_{X_j}$ by $\sum_{j=1}^N \delta_{-X_j}$ and $\theta$ by $-\theta>0$.

\subsubsection*{Small $|\lambda|$: Gaussian fluctuations.}
Our first result is an extension of Corollary 1.1 in \cite{Iksanov+Kabluchko:2016} to the complex case.
For $\lambda \in \Dom$ with $m(\lambda) \not = 0$, we set
\begin{equation}	\label{eq:sigma_lambda^2}
\sigma_\lambda^2	\defeq \E[|Z_1(\lambda)-1|^2] = \E[|Z_1(\lambda)|^2] - 1	\in [0,\infty].
\end{equation}
Notice that $\sigma_\theta^2 < \infty$ implies $\sigma_\lambda^2 < \infty$
since $|Z_1(\lambda)| \leq \frac{m(\theta)}{|m(\lambda)|} Z_1(\theta)$.

Throughout the paper, we call a complex random variable $\zeta = \xi+\imag \tau$
with $\xi = \Real(\zeta)$ and $\tau = \Imag(\zeta)$ standard normal
if $\xi$ and $\tau$ are independent, identically distributed centered normal random variables
with $\E[|\zeta|^2] = \E[\xi^2] + \E[\tau^2] = 1$.

\begin{theorem}[Gaussian case]	\label{Thm:Cor 1.1 in Iksanov+Kabluchko:2016}
Assume that $\lambda \in \Dom$ with $m(\lambda) \not = 0$ is such that $\sigma_\theta^2 < \infty$,
$\sigma_\lambda^2 > 0$ and $m(2\theta) < |m(\lambda)|^2$.
Define
\begin{equation*}
m = \begin{cases}
m(2\theta)		&	\text{if } |m(2\lambda)| < m(2\theta),		\\
m(2\lambda)	&	\text{if } |m(2\lambda)| = m(2\theta).
\end{cases}
\end{equation*}
Then
\begin{equation}	\label{eq:Cor 1.1 in Iksanov+Kabluchko:2016 complex}
\Law\big({\textstyle \frac{m(\lambda)^n}{m^{n/2}}}(Z(\lambda)\!-\!Z_n(\lambda)) \big| \F_n\big)\weaklyto
\Law\Big({\textstyle \frac{\sigma_\lambda}{\sqrt{1-m(2\theta)/|m(\lambda)|^2}}} \sqrt{Z(2 \theta)} X \Big| \F_\infty\Big)	\text{ in $\Prob$-probability}
\end{equation}
where $X$ is independent of $\F_\infty$.
Here, $X$ is complex standard normal if $|m(2\lambda)| < m(2\theta)$
whereas $X$ is real standard normal if $|m(2\lambda)| = m(2\theta)$.

If, additionally, either $2\theta \in \Lambda$ or $Z(2\theta)=0$ a.\,s., then the weak convergence in $\Prob$-probability in \eqref{eq:Cor 1.1 in Iksanov+Kabluchko:2016 complex}
can be strengthened to weak convergence $\Prob$-a.\,s.
\end{theorem}

A perusal of the proof of Theorem \ref{Thm:Cor 1.1 in Iksanov+Kabluchko:2016} reveals that the theorem still holds when $\cZ(\R)=\infty$ with positive probability,
that is, our standing assumption $\cZ(\R)<\infty$ a.\,s.\ is not needed for this result.

Further, notice that the limit in Theorem \ref{Thm:Cor 1.1 in Iksanov+Kabluchko:2016} may vanish a.\,s.,
namely, if $Z(2\theta)=0$ a.\,s.
Equivalent conditions for $(Z_n(2\theta))_{n \in \N_0}$ to be uniformly integrable
or equivalently
\begin{equation}\label{eq:nonzero}
\Prob(Z(2\theta)>0) > 0
\end{equation}
are given in
\cite{Lyons:1997} and \cite[Theorem 1.3]{Alsmeyer+Iksanov:2009}.
For instance,
\begin{equation}	\label{eq:Lyons' conditions}
\E[Z_1(2\theta) \log_+(Z_1(2\theta))] < \infty
\quad	\text{and}	\quad
2\theta m'(2\theta)/m(2\theta) < \log(m(2\theta))
\end{equation}
imply \eqref{eq:nonzero}. In particular, the condition $2\theta \in \Lambda$ comfortably ensures \eqref{eq:nonzero}.

However,
there may be a region of $\lambda \in \Lambda$ for which $m(2\theta) < |m(\lambda)|^2$ and $\sigma_\theta^2 < \infty$
but $Z(2\theta)$ is degenerate at $0$ as it is the case in Example \ref{Exa:binary Gaussian}\footnote{
In the example, the corresponding region is $\theta^2 + \eta^2 < \log 2$ and $\theta \geq \frac12\sqrt{2 \log 2}$.
}.
In this situation,
the assertion of Theorem \ref{Thm:Cor 1.1 in Iksanov+Kabluchko:2016} holds
but the limit is degenerate at $0$.
This means that $2\theta \not \in \Lambda$.
Typically, there is a real parameter $\vartheta>0$ on the boundary of $\Lambda$ such that
\begin{equation}	\label{eq:Lyons' conditions violated}
\vartheta m'(\vartheta)/m(\vartheta) = \log(m(\vartheta))
\end{equation}
and either $2\theta=\vartheta$ or $2\theta > \vartheta$.
The second case leads to a non-Gaussian regime in which the extremal positions dominate the fluctuations on $Z_n(\lambda)$ around $Z(\lambda)$.
This case will be dealt with further below.
In the first case, under mild moment assumptions,
a polynomial correction factor is required and a different martingale limit figures,
namely, the limit of the \emph{derivative martingale}.
More precisely, we suppose that \eqref{eq:Lyons' conditions} is violated because $2\theta = \vartheta$
where $\vartheta > 0$ is as in \eqref{eq:Lyons' conditions violated}.
Then, for $n \in \N_0$ and $u \in \Gen_n$, we define
\begin{equation}	\label{eq:V(u)}
V(u) \defeq \vartheta S(u) + n \log (m(\vartheta)).
\end{equation}
By definition and by \eqref{eq:Lyons' conditions violated},
\begin{equation}	\label{eq:normalized_V}
\E\bigg[\sum_{|u|=1}e^{-V(u)}\bigg]=1
\quad	\text{and}	\quad
\E\bigg[\sum_{|u|=1}V(u)e^{-V(u)}\bigg]=0.
\end{equation}
The branching random walk $((V(u))_{u \in \Gen_n})_{n \geq 0}$ is said to be in the \emph{boundary case}.
Then $W_n \defeq \sum_{|u|=n} e^{-V(u)} = Z_n(\vartheta) \to 0$ a.\,s.,
but the derivative martingale
\begin{equation}	\label{eq:derivative martingale}
\partial W_n	\defeq	\sum_{|u|=n} e^{-V(u)} V(u)
\end{equation}
converges $\Prob$-a.\,s.\ under appropriate assumptions
to some random variable $D_\infty$ satisfying $D_\infty > 0$ a.\,s.\ on the survival set $\Surv$, see \cite{Biggins+Kyprianou:2004} for details.
Due to a result by A{\"{\i}}d{\'e}kon and Shi \cite[Theorem 1.1]{Aidekon+Shi:2014},
the limit $D_\infty$ also appears as the limit in probability of the rescaled martingale $W_n$, namely,
\begin{align}	\label{eq:Aidekon+Shi}
\sqrt{n} W_n \Probto \sqrt{\frac{2}{\pi \sigma^2}} D_\infty
\end{align}
where
\begin{equation}	\label{eq:sigma^2<infty Aidekon+Shi}
\sigma^2 = \E\bigg[\sum_{|u|=1}V(u)^2 e^{-V(u)}\bigg] \in (0,\infty).
\end{equation}
Relation \eqref{eq:Aidekon+Shi} holds subject to the conditions \eqref{eq:normalized_V}, \eqref{eq:sigma^2<infty Aidekon+Shi} and
\begin{equation}	\label{eq:Aidekon+Shi condition}
\E [W_1\log_+^2(W_1)]<\infty	\quad	\text{and}	\quad	\E [\tilde{W}_1\log_+(\tilde{W}_1)]<\infty
\end{equation}
where $\tilde{W}_1 \defeq \sum_{|u|=1} e^{-V(u)} V(u)_+$ and $x_\pm\defeq \max(\pm x,0)$.
For the case where \eqref{eq:Aidekon+Shi} holds, we have the following result.

\begin{theorem}[Gaussian boundary case]		\label{Thm:Gaussian boundary case}
Suppose that $\vartheta > 0$ satisfies \eqref{eq:Lyons' conditions violated}
and that \eqref{eq:normalized_V}, \eqref{eq:sigma^2<infty Aidekon+Shi} and \eqref{eq:Aidekon+Shi condition} hold
for $V(u) = \vartheta S(u) + |u|\log(m(\vartheta))$, $u \in \Gen$.
Further, assume that $\lambda \in \Dom$ with $m(\lambda) \not = 0$ is such that $\sigma_\theta^2 < \infty$,
$\sigma_\lambda^2 > 0$, $m(2\theta) < |m(\lambda)|^2$ and $2\theta = \vartheta$.
Define
\begin{equation*}
m = \begin{cases}
m(2\theta)		&	\text{if } |m(2\lambda)| < m(2\theta),		\\
m(2\lambda)	&	\text{if } |m(2\lambda)|=m(2\theta)
\end{cases}
\end{equation*}
and $a_n \defeq n^{1/4} \frac{m(\lambda)^n}{m^{n/2}}$ for $n \in \N$.
Then
\begin{align}	\label{eq:Gaussian boundary case}
\Law\big(a_n (Z(\lambda)-Z_n(\lambda)) \big| \F_n\big)
\weaklyto
\Law\Big({\textstyle \frac{\sqrt{\frac{2}{\pi}} \frac{\sigma_\lambda}{\sigma}}{\sqrt{1-m(2\theta)/|m(\lambda)|^2}}} \sqrt{D_\infty} X \Big| \F_\infty\Big)
\quad	\text{in $\Prob$-probability}
\end{align}
where $X$ is independent of $\F_\infty$.
Here, $X$ is complex standard normal if $|m(2\lambda)| < m(2\theta)$
whereas $X$ is real standard normal if $|m(2\lambda)| = m(2\theta)$.
\end{theorem}

\subsubsection*{The regime in which the extremal positions dominate.}
Again suppose that $\vartheta > 0$ satisfies \eqref{eq:Lyons' conditions violated}
and that \eqref{eq:normalized_V}, \eqref{eq:sigma^2<infty Aidekon+Shi} and \eqref{eq:Aidekon+Shi condition} hold
for $V(u) = \vartheta S(u) + |u|\log(m(\vartheta))$,
$u \in \Gen$. Further, assume that $\lambda \in \Lambda$, but $2\theta > \vartheta$.
Then, typically, $Z_n(2\theta) \to 0$ because \eqref{eq:Lyons' conditions} is violated because
\begin{equation*}
2\theta m'(2\theta)/m(2\theta) > \log(m(2\theta)).
\end{equation*}
It is known, see e.g.\ \cite{Shi:2015}, that $\min_{|u|=n} V(u)$ is of the order $\frac32 \log n$ as $n \to \infty$.
It will turn out that this is too slow for a result in the spirit of Theorem \ref{Thm:Cor 1.1 in Iksanov+Kabluchko:2016}
in the sense that the contributions of the particles with small positions in the $n$th generation to $Z(\lambda)-Z_n(\lambda)$ are substantial,
and hence no (conditionally) infinitely divisible limit distribution can be expected.
Instead, the description of the fluctuations $Z(\lambda)-Z_n(\lambda)$ will follow from Madaule's work \cite{Madaule:2017},
where the behavior of the point processes
\begin{equation*}
\mu_n \defeq \sum_{|u|=n} \delta_{V_n(u)}
\end{equation*}
with $V_n(u) \defeq V(u) - \tfrac32 \log n$ was studied.
For the reader's convenience, we state in detail a consequence of the main result in \cite{Madaule:2017}.

\begin{proposition}	\label{Prop:Madaule's theorem}
Suppose the branching random walk $(V(u))_{u \in \Gen}$
satisfies \eqref{eq:normalized_V}, \eqref{eq:sigma^2<infty Aidekon+Shi} and \eqref{eq:Aidekon+Shi condition}.
Further, suppose that
\begin{align*}	\tag{A1}\label{ass:non-latice}
\text{The branching random walk $(V(u))_{u \in \Gen}$ is non-lattice.}
\end{align*}
Then there is a point process $\mu_{\infty} = \sum_{k \in \N} \delta_{P_k}$ such that $\mu_{\infty}((-\infty,0])$ is a.\,s.\ finite and $\mu_n$ converges in distribution to $\mu_{\infty}$
(in the space of locally finite point measures equipped with the topology of vague convergence).
\end{proposition}
\begin{proof}[Source]
This can be derived from \cite[Theorem 1.1]{Madaule:2017}.
\end{proof}

Let $Z^{(1)}(\lambda), Z^{(2)}(\lambda), \ldots$ denote independent random variables with the same distribution
as $Z(\lambda)-1$ which are independent of $\mu_\infty$.
We consider the following series
\begin{align}	\label{eq:extremal sum}
X_{\mathrm{ext}}\defeq \sum_{k}e^{-\frac{\lambda P_k^\ast}{\vartheta}}Z^{(k)}(\lambda)=\lim_{n\to\infty}\sum_{k=1}^n e^{-\frac{\lambda P_k^\ast}{\vartheta}}Z^{(k)}(\lambda),
\end{align}
where $-\infty < P_1^\ast \leq P_2^\ast \leq \ldots$ are the atoms of $\mu_\infty$ arranged in increasing order.
\begin{theorem}[Domination by extremal positions]		\label{Thm:domination of extremal particles}
Suppose that $\vartheta > 0$ satisfies \eqref{eq:Lyons' conditions violated}
and that \eqref{eq:normalized_V}, \eqref{eq:sigma^2<infty Aidekon+Shi}, \eqref{eq:Aidekon+Shi condition} and \eqref{ass:non-latice} hold
for $V(u) = \vartheta S(u) + |u|\log(m(\vartheta))$, $u \in \Gen$.
Let $\lambda \in \Lambda$ and assume that $\theta \in (\frac\vartheta2,\vartheta)$.
If there is $p \in (\frac{\vartheta}{\theta},2]$ satisfying $\E[|Z(\lambda)|^p]\!\!~<~\!\!\infty$,
then the series $X_{\mathrm{ext}}$ defined by \eqref{eq:extremal sum} converges a.\,s.\ to a non-degenerate limit. Moreover,
\begin{equation*}
n^{\tfrac{3\lambda}{2 \vartheta}} \Big(\frac{m(\lambda)}{m(\vartheta)^{\lambda/\vartheta}}\Big)^{\!n} (Z(\lambda)-Z_n(\lambda))
\distto X_{\mathrm{ext}}.
\end{equation*}
\end{theorem}

Sufficient conditions for $\E[|Z(\lambda)|^p]<\infty$, which are easy to check, are \eqref{eq:gamma moment Z_1(theta)} and \eqref{eq:contraction condition}.
Further sufficient conditions for $\E[|Z(\lambda)|^p]<\infty$ are given in Proposition \ref{Prop:heavy tail bound Z(lambda)} below.
Finally, we should mention the upcoming paper \cite{Iksanov+al:2018+} in which conditions for the convergence in $L^p$ for $(Z_n(\lambda))_{n \in \N_0}$ are provided.

\subsubsection*{The boundary of $\Lambda$: Stable fluctuations.}

It has been shown in \cite[Theorem 2.1]{Kolesko+Meiners:2017}
that the martingale $Z_n(\lambda)$ converges on a part of the boundary of $\Lambda$.
More precisely, consider the condition
\begin{equation*}	\tag{C1}	\label{eq:C1}
\frac{m(\alpha\theta)}{|m(\lambda)|^\alpha} = 1
\quad	\text{and}	\quad
\frac{\theta m'(\theta \alpha)}{|m(\lambda)|^\alpha} = \log(|m(\lambda)|)
\end{equation*}and define
\begin{equation*}
\partial \Lambda^{(1,2)}	\defeq	\{\lambda \in \partial \Lambda \cap \Dom: \text{\eqref{eq:C1} holds with } \alpha \in (1,2)\}.
\end{equation*}
Theorem 2.1 in \cite{Kolesko+Meiners:2017} says that if $\lambda \in \Dom$ satisfies \eqref{eq:C1}
(actually, Theorem 2.1 in \cite{Kolesko+Meiners:2017} requires a weaker assumption)
and if $\E[|Z_1(\lambda)|^\alpha \log_+^{2+\epsilon}(|Z_1(\lambda)|)] < \infty$ for some $\epsilon > 0$,
then $(Z_n(\lambda))_{n \in \N_0}$ converges a.\,s.\ and in $L^p$ for every $p<\alpha$ to some limit $Z(\lambda)$ satisfying $\E[Z(\lambda)]=1$.
If an additional moment assumption holds, then a simplified version of the proof of Theorem 2.1 in \cite{Kolesko+Meiners:2017} gives the following result.

\begin{proposition}	\label{Prop:heavy tail bound Z(lambda)}
Suppose that $\lambda \in \Dom$ satisfies
\begin{equation*}
\frac{m(\alpha\theta)}{|m(\lambda)|^\alpha} = 1
\quad	\text{and}	\quad
\frac{\theta m'(\theta \alpha)}{|m(\lambda)|^\alpha} \leq \log(|m(\lambda)|)
\end{equation*}
for some $\alpha \in (1,2)$. If, additionally, $\E[|Z_1(\lambda)|^{\gamma}]<\infty$ for some $\alpha<\gamma \leq 2$,
then $Z_n(\lambda) \to Z(\lambda)$ in $L^{p}$ for all $p<\alpha$ and there exists a constant $C > 0$ such that
\begin{equation}	\label{eq:heavy tail bound Z(lambda)}
\Prob(|Z(\lambda)| \geq t) \leq C t^{-\alpha}
\end{equation}
for all $t > 0$.
\end{proposition}

For the rest of this section, we assume that $\lambda \in \partial \Lambda^{(1,2)}$ and that $\alpha \in (1,2)$ satisfies \eqref{eq:C1}.
Notice that if $\vartheta > 0$ is defined via \eqref{eq:Lyons' conditions violated},
then $\alpha \theta = \vartheta$ in the given situation.
To determine the fluctuations of $Z_n(\lambda)$ around $Z(\lambda)$ in this setting,
we require stronger assumptions than those of Theorem 2.1 in \cite{Kolesko+Meiners:2017}.
First of all, as before, we define $V(u)$ via \eqref{eq:V(u)}, i.e., $V(u) \defeq \vartheta S(u) + n \log (m(\vartheta))$
for $n \in \N_0$ and $u \in \Gen_n$.
Then \eqref{eq:C1} becomes \eqref{eq:normalized_V}.
Further, we shall require that the following conditions hold:
\begin{align}
\E[Z_1(\theta)^\gamma] &< \infty\text{ for some } \gamma \in (\alpha,2],		\label{ass:Lgamma}	\\
\E\big[Z_1(\kappa \theta)^2] &< \infty\text{ for some } \kappa \in (\tfrac\alpha2,1).	\label{ass:L2_locally}
\end{align}
We denote by $\gr=\gr(\lambda)$ the smallest closed  subgroup of
the multiplicative group $\C^* = \C \setminus \{0\}$
such that
\begin{equation*}
\Prob\Big(\frac{e^{-\lambda S(u)}}{m(\lambda)}\in \gr \text{ for all } u \in \I\Big)=1.
\end{equation*}
Furthermore, for simplicity of presentation, we assume that
\begin{align}	\label{ass:nonarithmeticity}
\{|z|:z\in \gr \} =\Rp \defeq (0,\infty).
\end{align}
Let us now briefly describe the structure of $\gr$.
If the subgroup $\gr_1=\gr\cap\{|z|=1\}$ coincides with the unit sphere $\{|z|=1\}$, then $\gr$ is the whole multiplicative group $\C^*$.
Otherwise, $\gr_1$ is a finite group and $\gr$ consists of finitely many connected components.
By $\gr_{\R}$ we denote the one-parameter subgroup of $\gr$ which is either $\R_{>}$
if $\gr=\C^*$ or it is the connected component of $\gr$ that contains $1$ if $\gr \neq \C^*$.
Clearly, $\gr_{\R}$ is a subgroup isomorphic to the multiplicative group $\R_{>}$.
By $\gamma_t$ we denote the canonical parametrization of $\gr_\R$ satisfying $|\gamma_t|=t$.
We infer that there exists some $w \in \C$ with $\Real(w)=1$ such that
\begin{equation}	\label{eq:w}
\gamma_t = t^w = \exp(w \log t)		\text{ for all }	t>0.
\end{equation}
Clearly, $w=1$ if $\gr_\R = \R_>$.
It is also worth mentioning
that $\gr_1 \times \R_{>} \simeq \gr$ via the isomorphism $T: \gr_1 \times \R_{>} \to \gr$, $(z,t) \mapsto z \gamma_t = z t^w$.
For illustration purposes, we interrupt the setup and discuss an example.

\begin{example}[Binary splitting and Gaussian increments revisited]	\label{Exa:binary Gaussian 2}
Again we consider a branching random walk with binary splitting and independent standard Gaussian increments,
that is, $\cZ = \delta_{X_1}+\delta_{X_2}$
with independent standard normal random variables $X_1,X_2$.
Recall that, in this situation, we have $m(\lambda) = 2 \exp(\lambda^2/2)$ for $\lambda \in \C$.
The parameter region we are interested in is
$\sqrt{2 \log 2}/2 < \theta < \sqrt{2 \log 2}$ and $\eta=\sqrt{2 \log 2}-\theta$, see Figure \ref{fig:Gaussian Lambda}.
For $\lambda = \theta+\imag\eta$ from this region, we have
\begin{equation*}
m(\lambda) = 2 \exp((\theta+\imag \eta)^2/2) = \exp(\sqrt{2 \log 2} \theta + \imag \theta\eta).
\end{equation*}
Therefore, $\gr$ is generated by the set
\begin{equation*}
\{e^{\theta (x-\sqrt{2 \log 2}) + \imag \eta (x-\theta)}: x \in \R\}
\end{equation*}
which we may rewrite as
\begin{equation*}
\{e^{\imag \eta^2} \cdot e^{(\theta+\imag \eta) x}: x \in \R\}.
\end{equation*}
In particular, \eqref{ass:nonarithmeticity} holds. Moreover, $\gr_1$ is the closed (multiplicative) subgroup of the unit circle
generated by $e^{\imag \eta^2}$.
This group is finite if and only if $\frac1{2\pi}\eta^2\in\Q$, and $\gr_1 = \{z \in \C: |z|=1\}$, otherwise.
As $\theta$ varies over $(\sqrt{2 \log 2}/2,\sqrt{2 \log 2})$, the square of the imaginary part, $\eta^2$, ranges over the whole interval $(0,\frac12 \log 2)$.
Thus, for all but countably many $\theta$, the group $\gr$ equals $\C^*$, but for countably many $\theta$, $\gr$ will consist of a finite family of `snails'
as depicted in the figure below.
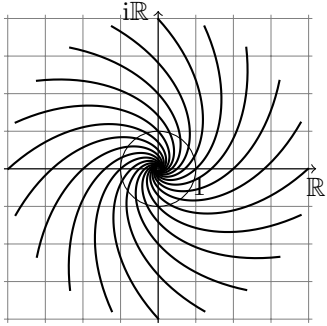
\begin{figure}[h]
\begin{center}
\begin{tikzpicture}[scale=0.5]
		\draw [help lines] (-4.1,-4.1) grid (4.1,4.1);
		\draw (1.1,0) node[below]{$1$};
		\draw [->] (-4.1,0) -- (4.2,0) node[below] {$\R$}; 
		\draw [->] (0,-4.1) -- (0,4.2) node[left] {$\imag\R$}; ;
		\draw (0,0) circle (1);
\draw [thick, domain=-5:2.25, samples=100] plot ({exp((sqrt(2*ln(2)) - sqrt(pi/10))*\x)*cos(deg(pi/10+sqrt(pi/10)*\x))},{exp((sqrt(2*ln(2)) - sqrt(pi/10))*\x)*sin(deg(pi/10+sqrt(pi/10)*\x))});
\draw [thick, domain=-5:2.25, samples=100] plot ({exp((sqrt(2*ln(2)) - sqrt(pi/10))*\x)*cos(deg(2*pi/10+sqrt(pi/10)*\x))},{exp((sqrt(2*ln(2)) - sqrt(pi/10))*\x)*sin(deg(2*pi/10+sqrt(pi/10)*\x))});
\draw [thick, domain=-5:2.25, samples=100] plot ({exp((sqrt(2*ln(2)) - sqrt(pi/10))*\x)*cos(deg(3*pi/10+sqrt(pi/10)*\x))},{exp((sqrt(2*ln(2)) - sqrt(pi/10))*\x)*sin(deg(3*pi/10+sqrt(pi/10)*\x))});
\draw [thick, domain=-5:2.25, samples=100] plot ({exp((sqrt(2*ln(2)) - sqrt(pi/10))*\x)*cos(deg(4*pi/10+sqrt(pi/10)*\x))},{exp((sqrt(2*ln(2)) - sqrt(pi/10))*\x)*sin(deg(4*pi/10+sqrt(pi/10)*\x))});
\draw [thick, domain=-5:2.25, samples=100] plot ({exp((sqrt(2*ln(2)) - sqrt(pi/10))*\x)*cos(deg(5*pi/10+sqrt(pi/10)*\x))},{exp((sqrt(2*ln(2)) - sqrt(pi/10))*\x)*sin(deg(5*pi/10+sqrt(pi/10)*\x))});
\draw [thick, domain=-5:2.25, samples=100] plot ({exp((sqrt(2*ln(2)) - sqrt(pi/10))*\x)*cos(deg(6*pi/10+sqrt(pi/10)*\x))},{exp((sqrt(2*ln(2)) - sqrt(pi/10))*\x)*sin(deg(6*pi/10+sqrt(pi/10)*\x))});
\draw [thick, domain=-5:2.25, samples=100] plot ({exp((sqrt(2*ln(2)) - sqrt(pi/10))*\x)*cos(deg(7*pi/10+sqrt(pi/10)*\x))},{exp((sqrt(2*ln(2)) - sqrt(pi/10))*\x)*sin(deg(7*pi/10+sqrt(pi/10)*\x))});
\draw [thick, domain=-5:2.25, samples=100] plot ({exp((sqrt(2*ln(2)) - sqrt(pi/10))*\x)*cos(deg(8*pi/10+sqrt(pi/10)*\x))},{exp((sqrt(2*ln(2)) - sqrt(pi/10))*\x)*sin(deg(8*pi/10+sqrt(pi/10)*\x))});
\draw [thick, domain=-5:2.25, samples=100] plot ({exp((sqrt(2*ln(2)) - sqrt(pi/10))*\x)*cos(deg(9*pi/10+sqrt(pi/10)*\x))},{exp((sqrt(2*ln(2)) - sqrt(pi/10))*\x)*sin(deg(9*pi/10+sqrt(pi/10)*\x))});
\draw [thick, domain=-5:2.25, samples=100] plot ({exp((sqrt(2*ln(2)) - sqrt(pi/10))*\x)*cos(deg(10*pi/10+sqrt(pi/10)*\x))},{exp((sqrt(2*ln(2)) - sqrt(pi/10))*\x)*sin(deg(10*pi/10+sqrt(pi/10)*\x))});
\draw [thick, domain=-5:2.25, samples=100] plot ({exp((sqrt(2*ln(2)) - sqrt(pi/10))*\x)*cos(deg(11*pi/10+sqrt(pi/10)*\x))},{exp((sqrt(2*ln(2)) - sqrt(pi/10))*\x)*sin(deg(11*pi/10+sqrt(pi/10)*\x))});
\draw [thick, domain=-5:2.25, samples=100] plot ({exp((sqrt(2*ln(2)) - sqrt(pi/10))*\x)*cos(deg(12*pi/10+sqrt(pi/10)*\x))},{exp((sqrt(2*ln(2)) - sqrt(pi/10))*\x)*sin(deg(12*pi/10+sqrt(pi/10)*\x))});
\draw [thick, domain=-5:2.25, samples=100] plot ({exp((sqrt(2*ln(2)) - sqrt(pi/10))*\x)*cos(deg(13*pi/10+sqrt(pi/10)*\x))},{exp((sqrt(2*ln(2)) - sqrt(pi/10))*\x)*sin(deg(13*pi/10+sqrt(pi/10)*\x))});
\draw [thick, domain=-5:2.25, samples=100] plot ({exp((sqrt(2*ln(2)) - sqrt(pi/10))*\x)*cos(deg(14*pi/10+sqrt(pi/10)*\x))},{exp((sqrt(2*ln(2)) - sqrt(pi/10))*\x)*sin(deg(14*pi/10+sqrt(pi/10)*\x))});
\draw [thick, domain=-5:2.25, samples=100] plot ({exp((sqrt(2*ln(2)) - sqrt(pi/10))*\x)*cos(deg(15*pi/10+sqrt(pi/10)*\x))},{exp((sqrt(2*ln(2)) - sqrt(pi/10))*\x)*sin(deg(15*pi/10+sqrt(pi/10)*\x))});
\draw [thick, domain=-5:2.25, samples=100] plot ({exp((sqrt(2*ln(2)) - sqrt(pi/10))*\x)*cos(deg(16*pi/10+sqrt(pi/10)*\x))},{exp((sqrt(2*ln(2)) - sqrt(pi/10))*\x)*sin(deg(16*pi/10+sqrt(pi/10)*\x))});
\draw [thick, domain=-5:2.25, samples=100] plot ({exp((sqrt(2*ln(2)) - sqrt(pi/10))*\x)*cos(deg(17*pi/10+sqrt(pi/10)*\x))},{exp((sqrt(2*ln(2)) - sqrt(pi/10))*\x)*sin(deg(17*pi/10+sqrt(pi/10)*\x))});
\draw [thick, domain=-5:2.25, samples=100] plot ({exp((sqrt(2*ln(2)) - sqrt(pi/10))*\x)*cos(deg(18*pi/10+sqrt(pi/10)*\x))},{exp((sqrt(2*ln(2)) - sqrt(pi/10))*\x)*sin(deg(18*pi/10+sqrt(pi/10)*\x))});
\draw [thick, domain=-5:2.25, samples=100] plot ({exp((sqrt(2*ln(2)) - sqrt(pi/10))*\x)*cos(deg(19*pi/10+sqrt(pi/10)*\x))},{exp((sqrt(2*ln(2)) - sqrt(pi/10))*\x)*sin(deg(19*pi/10+sqrt(pi/10)*\x))});
\draw [thick, domain=-5:2.25, samples=100] plot ({exp((sqrt(2*ln(2)) - sqrt(pi/10))*\x)*cos(deg(20*pi/10+sqrt(pi/10)*\x))},{exp((sqrt(2*ln(2)) - sqrt(pi/10))*\x)*sin(deg(20*pi/10+sqrt(pi/10)*\x))});
		\end{tikzpicture}
\caption{The group $\gr$ in the case $\theta=\sqrt{2 \log 2} - \sqrt{\frac\pi{10}}$ and $\eta = \sqrt{\frac\pi{10}}$.}
\label{fig:snails}
\end{center}
\end{figure}
Finally, whenever $\gr = \C^*$, the scaling exponent can be chosen as $w=1$.
When $\gr \not = \C^*$, then $\gr_1$ is finite and the connected component of $\gr_1$ containing $1$ is
\begin{equation*}
\{e^{(\theta+\imag \eta) x}: x \in \R\} = \{e^{\lambda x}: x \in \R\} = \{t^{\lambda/\theta}: t>0\}
\end{equation*}
so that $w=\lambda/\theta$ in this case.
\end{example}

By $\varrho$ we denote the Haar measure on $\gr$ satisfying the normalization condition
\begin{equation}	\label{eq:Haar normalization}
\varrho(\{z \in \C: 1 \leq |z| < e\}) = 1,
\end{equation}
i.e., $\varrho$ is the image of the measure $\varrho_1\times\frac{\dt}t$, where $\varrho_1$ is the uniform distribution on $\gr_1$, via the isomorphism $T$.

To understand the fluctuations of $Z(\lambda)-Z_n(\lambda)$, one needs to know the tail behavior of $Z(\lambda)$.
The following theorem, which is interesting in its own right, provides the information required.
For its formulation, we introduce some {additional} notation.
We write $\Chat = \C \cup \{\infty\}$ for the one-point (Alexandroff) compactification of $\C$.
Further, we denote by $C_{\mathrm{c}}^{2}(\Chat \setminus \{0\})$
the set of real-valued, twice continuously partially differentiable functions on $\Chat \setminus \{0\}$ with compact support.
Finally, we remind the reader that a measure $\nu$ on $\C$ is called a \emph{L\'evy measure}
if $\nu(\{0\})=0$ and $\int_\C (|z|^2 \wedge 1) \, \nu(\dz) < \infty$.
A L\'evy measure $\nu$ is called $(\gr,\alpha)$-invariant
if $\nu(uB) = |u|^{-\alpha} \nu(B)$ for all $u \in \gr$ and all Borel sets $B \subseteq \C \setminus \{0\}$.
The L\'evy measure $\nu$ is called non-zero if $\nu(B)>0$ for some Borel set $B$ as above.

\begin{theorem}	\label{Thm:tail Z(lambda) in Lambda^(1,2)}
Let $\lambda \in \Dom$ satisfy \eqref{eq:C1} with $\alpha \in (1,2)$.
Further, suppose that \eqref{ass:Lgamma}, \eqref{ass:L2_locally} and \eqref{ass:nonarithmeticity} hold.
Then there is a non-zero
$(\gr,\alpha)$-invariant L\'evy measure $\nu$ on $\C$
such that
\begin{equation*}
\lim_{\substack{|z|\to 0, \\ z\in \gr}}|z|^{-\alpha}\E[\phi(zZ(\lambda))]=	{\textstyle \int \phi \, \dnu}
\end{equation*}
for all $\phi \in C_{\mathrm{c}}^{2}(\Chat \setminus \{0\})$.
\end{theorem}

We denote by $(X_t)_{t \geq 0}$ a complex-valued L\'evy process
which is independent of $\F_\infty$ and has characteristic exponent
\begin{equation*}
\Psi(x) = \int \big(e^{\imag \langle x,z\rangle}-1-\imag  \langle x,z\rangle\big) \, \nu(\dz),	\quad	x \in \C.
\end{equation*}
Notice that $\Psi$ is well-defined as integration by parts gives
\begin{align*}
\int_{\{|z| \geq 1\}}(|z|-1) \, \nu(\dz) = \int_1^\infty \nu(\{|z| \geq t\}) \, \dt = \nu(\{|z| \geq 1\}) \, \int_1^\infty \! t^{-\alpha} \, \dt < \infty.
\end{align*}
Therefore, $(X_t)_{t \geq 0}$ is the L\'evy process associated with the L\'evy-Khintchine characteristics
$(0,-\int_{\{|z|>1\}} z \, \nu(\dz),\nu)$ (cf.\ \cite[p.\;291, Corollary 15.8]{Kallenberg:2002}).

Now we are ready to describe the fluctuations of $Z(\lambda)-Z_n(\lambda)$ for $\lambda \in \partial \Lambda^{(1,2)}$.
\begin{theorem}	\label{Thm:fluctuations on partial Lambda^(1,2)}
Suppose that the assumptions of Theorem \ref{Thm:tail Z(lambda) in Lambda^(1,2)} hold.
Then there exists $w\in\C$ such that $\Real(w)=1$ (see \eqref{eq:w} for the definition of $w$) and
\begin{equation}	\label{eq:Lambda^(1,2) weakly in probability}
\Law\big(n^{\frac{w}{2\alpha}}(Z(\lambda)-Z_n(\lambda)) \mid \mathcal F_n\big) \weaklyto \Law(X_{c D_{\infty}}|\mathcal F_{\infty})	\quad \text{in $\Prob$-probability}
\end{equation}
for $c \defeq \sqrt{\frac2{\pi \sigma^2}}$, $\sigma^2$ defined by
\eqref{eq:sigma^2<infty Aidekon+Shi} with $V(u)$ as in \eqref{eq:V(u)}
and $D_\infty$ being the a.\,s.\ limit of the derivative martingale defined in \eqref{eq:derivative martingale}.
\end{theorem}

\subsubsection*{Related literature.}

The martingale convergence theorem guarantees the almost sure convergence of $Z_n(\theta)$,
but its limit $Z(\theta)$ may vanish a.\,s.
Equivalent conditions for $\Prob(Z(\theta)=0) < 1$ can be found in \cite[Lemma 5]{Biggins:1977}, \cite{Lyons:1997} and \cite[Theorem 1.3]{Alsmeyer+Iksanov:2009}.
Convergence in distribution of $a_n (Z(\lambda)-Z_n(\lambda))$ as $n \to \infty$ for constants $a_n > 0$
can be viewed as a result on the rate of convergence.
In \cite{Alsmeyer+al:2009,Iksanov+Kabluchko:2016,Iksanov+al:2018,Iksanov+Meiners:2010,Iksanov:2006}
the rate of convergence of $Z_n(\theta)$ to $Z(\theta)$ has been investigated
in the regime $\Prob(Z(\theta)=0) < 1$.
The papers \cite{Alsmeyer+al:2009,Iksanov+Meiners:2010,Iksanov:2006}
deal with the issue of convergence of the infinite series
\begin{equation}	\textstyle	\label{eq:infinite series}
\sum_{n=0}^\infty a_n (Z(\theta)-Z_n(\theta)).
\end{equation}
More precisely, in \cite{Alsmeyer+al:2009}, necessary conditions
and sufficient conditions for the convergence in $L^p$ of
the infinite series in \eqref{eq:infinite series} are given in the situation where $a_n = e^{an}$ for some $a > 0$.
Sufficient conditions for the almost sure convergence of the series in \eqref{eq:infinite series}
have been provided in the case where $a_n = e^{an}$ for some $a > 0$ in \cite{Iksanov+Meiners:2010}
and in the case where $(a_n)_{n \in \N_0}$ is regularly varying at $+\infty$ in \cite{Iksanov:2006}.

The papers \cite{Iksanov+Kabluchko:2016,Iksanov+al:2018} are in the spirit of the article at hand.
In these works,
for $\alpha \in (1,2]$, it is shown that if $\kappa \defeq m(\alpha \theta) / m(\theta)^\alpha < 1$,
then $\kappa^{-n/\alpha} (Z(\theta)-Z_n(\theta))$ converges in distribution
to a random variable $Z(\alpha \theta)^{1/\alpha} U$
where $U$ is a (non-degenerate) centered $\alpha$--stable random variable (normal, if $\alpha=2$)
independent of $Z(\alpha \theta)$. Specifically,
the case where $\alpha \in (1,2)$ and $\Prob(Z_1(\theta) > x) \sim c x^{-\alpha}$ for some $c > 0$ is covered in \cite[Corollary 1.3]{Iksanov+al:2018}
whereas the case $\alpha=2$ and $\E[Z_1(\theta)^2]<\infty$ is investigated in \cite[Corollary 1.1]{Iksanov+Kabluchko:2016}.
Both papers actually contain functional versions of these convergences.
The aforementioned assertions are extensions of the corresponding results for Galton-Watson processes \cite{Buehler:1969,Heyde:1970,Heyde+Brown:1971}.

The counterpart of our Theorem \ref{Thm:Cor 1.1 in Iksanov+Kabluchko:2016},
which gives the fluctuations of Biggins' martingales for small parameters
has a natural analogue in the complex branching Brownian motion energy model.
The corresponding statement in the latter model is \cite[Theorem 1.4]{Hartung+Klimovsky:2017}.

It is well known that if $\theta m'(\theta)/m(\theta) = \log(m(\theta))$, then $Z_n(\theta)$ converges to $0$ a.\,s.
In this case a natural object to study is the derivative martingale $(D_n)_{n \in \N_0}$.
In order to study the fluctuations of $D_n$ around its limit $D_{\infty}$ one needs an additional correction term of order $(\log n)\slash\sqrt n$.
The corresponding result, again in the context of branching Brownian motion,
is given in \cite{Maillard+Pain:2018}, where it is shown that $\sqrt n(D_{\infty}-D_n+\frac{\log n}{\sqrt{2\pi n}}D_{\infty})\distto S_{D_{\infty}}$
for an independent 1--stable L\'evy process $(S_t)_{t \geq 0}$.

The martingale limits $Z(\lambda)$ solve smoothing equations, namely,
\begin{equation}	\label{eq:smoothing eq for Z(lambda)}
Z(\lambda) = \sum_{|u|=1} \frac{e^{-\lambda S(u)}}{m(\lambda)} [Z(\lambda)]_u	\quad	\text{a.\,s.}
\end{equation}
where the $[Z(\lambda)]_u$, $u \in \N$ are independent copies of $Z(\lambda)$ which are independent of the positions $S(u)$, $|u|=1$.
If $U$ is centered $\alpha$--stable and independent of $Z(\alpha \theta)$,
then the limit variable $Z(\alpha \theta)^{1/\alpha} U$ in \cite[Corollary 1.1]{Iksanov+Kabluchko:2016} and \cite[Corollary 1.3]{Iksanov+al:2018}
satisfies
\begin{equation*}
Z(\alpha \theta)^{1/\alpha} \, U = \bigg(\sum_{|v|=1} \frac{e^{-\alpha \theta S(v)}}{m(\alpha \theta)} [Z(\lambda)]_v\bigg)^{\!\!\frac{1}{\alpha}} \, U
\eqdist \sum_{|v|=1} \frac{e^{-\theta S(v)}}{m(\alpha \theta)^{1/\alpha}} [Z(\lambda)]_v^{1/\alpha} \, U_v
\end{equation*}
where $(U_v)_{v \in \N}$ is a family of independent copies of $U$
which is independent of all other random variables appearing on the right-hand side of the latter distributional equality.
Hence, the distribution
of $Z(\alpha \theta)^{1/\alpha} \, U$ is a solution to the following fixed-point equation of the smoothing transformation:
\begin{equation}	\label{eq:smoothing eq for Z(lambda)^1/alpha U}
X	\eqdist	\sum_{j \geq 1} T_j X_j
\end{equation}
where $T_j \defeq \1_{\{j \in \Gen_1\}} \frac{e^{-\theta S(j)}}{m(\alpha \theta)^{1/\alpha}}$ and the $X_j$, $j \in \N$ are independent
copies of the random variable $X$.
In \eqref{eq:smoothing eq for Z(lambda)^1/alpha U}, which should be seen as an equation for the distribution
of $X$ rather than the random variable $X$ itself,
$T_1,T_2,\ldots$ are considered given whereas the distribution
of $X$ is considered unknown.
Equation \eqref{eq:smoothing eq for Z(lambda)^1/alpha U} has been studied in depth in the case where the $T_j$ and $X_j$ are nonnegative,
see \cite{Alsmeyer+al:2012} for the most recent contribution and an overview of earlier results.
If, however, we consider complex $Z(\lambda)$ at complex parameters, \eqref{eq:smoothing eq for Z(lambda)} becomes an equation between complex random variables
and it is reasonable to conjecture that the limiting distributions
of $a_n (Z(\lambda)-Z_n(\lambda))$ are solutions to \eqref{eq:smoothing eq for Z(lambda)^1/alpha U}
with complex-valued $T_j$ and $X_j$.
A systematic study of \eqref{eq:smoothing eq for Z(lambda)^1/alpha U} in the case where $T_j$ and $X_j$ are complex-valued
has been addressed only recently in \cite{Meiners+Mentemeier:2017}.

\section{Preliminaries}	\label{sec:preliminaries}

In this section, we fix some notation and set the stage for the proofs of our main results.

\subsection{Notation}

\subsubsection*{Complex numbers.}
Throughout the paper, we identify $\C$ and $\R^2$.
For instance, for $z \in \C$, we sometimes write $z_1$ for $\Real(z)$ and $z_2$ for $\Imag(z)$.
Further, we sometimes identify $z \in \C$ with the column vector $(z_1,z_2)^\transp$
and write $z^\transp$ for the row vector $(z_1,z_2)$.
As usual, we write $\overline{z}$ for the complex conjugate of $z \in \C$, i.e., $\overline{z} = z_1 - \imag z_2$.
In some proofs, we identify a complex number $z = r e^{\imag \varphi}$
with the matrix $r R(\varphi)$ where $R(\varphi)$ is the $2 \times 2$ rotation matrix
\begin{equation*}
R(\varphi) =
\begin{pmatrix}
\cos \varphi	&	-\sin \varphi	\\
\sin \varphi 	&	\cos \varphi		
\end{pmatrix}.
\end{equation*}
By $\Chat$ we denote the one-point compactification of $\C$, i.e., $\Chat = \C \cup \{\infty\}$
and a set $K \subseteq \Chat$ is relatively compact if it is relatively compact in $\C$ or the complement of a bounded subset of $\C$.
A function $\phi:\Chat \to \R$ is differentiable at $\infty$ if $\psi:\C \to \R$ with $\psi(z) = \phi(1/z)$ for $z \not = 0$ and $\psi(0) = \phi(\infty)$
is differentiable at $0$.

\subsubsection*{Conditional expectations.}
Throughout the paper, we write
$\Prob_n(\cdot)$ for $\Prob(\cdot | \F_n)$ for every $n \in \N_0$.
The corresponding (conditional) expectation and variance are denoted $\E_n[\cdot] \defeq \E[\cdot | \F_n]$
and $\Var_n[\cdot] \defeq \Var[\cdot | \F_n]$.
We further write $\E[X; A]$ for $\E[X\1_{A}]$, $\Var[X; A]$ for $\Var[X\1_{A}]$,
and $\Cov[X; A]$ for the covariance matrix of the vector $X\1_{A}$.
If $X$ is a complex random variable, we write $\Cov[X]$ for the covariance matrix of the vector $(\Real(X), \Imag(X))^\transp$.
We also use the analogous notation with $\E$, $\Var$ and $\Cov$ replaced by $\E_n$, $\Var_n$ and $\Cov_n$.

\subsubsection*{The martingale.}
Further, when $\lambda \in \Lambda$ is fixed, we sometimes write $Z_n$ for $Z_n(\lambda)$ and $Z$ for $Z(\lambda)$
in order to unburden the notation.

\subsection{Background and relevant results from the literature}	\label{subsec:background}

\subsubsection*{Recursive decomposition of tail martingales.}
Throughout the paper, we denote by $[\cdot]_u$, $u \in \I$ the canonical shift operators,
that is, for any function $\Psi$ of $(\cZ(v))_{v \in \I}$,
we write $[\Psi]_u$ for the same function applied to the family $(\cZ(uv))_{v \in \I}$.
Using this notation, we obtain the following decomposition of $Z(\lambda)-Z_n(\lambda)$:
\begin{equation}	\label{eq:decomposition nth gen}
Z(\lambda)-Z_n(\lambda) = m(\lambda)^{-n} \sum_{|u|=n} e^{-\lambda S(u)} ([Z(\lambda)]_u-1)	\quad	\text{a.\,s.,}
\end{equation}
which is valid for every $n \in \N_0$. Therefore, with respect to $\Prob_n$,
$Z(\lambda)-Z_n(\lambda)$ is a sum of i.i.d.\ centered random variables.
This explains the appearance of (randomly scaled) normal or stable distributions in our main theorems.

\subsubsection*{Minimal position: First order.}
If $\theta > 0$ with $m(\theta)<\infty$, then \cite[Theorem 3]{Biggins:1998} gives
\begin{equation}	\label{eq:Biggins:1998}
\sup_{|u|=n} {\textstyle \frac{e^{-\theta S(u)}}{m(\theta)^{n}}} \to 0		\quad	\text{a.\,s.\ as } n \to \infty.
\end{equation}

\section{The Gaussian regime}	\label{sec:Gaussian regime}

Before we prove Theorems \ref{Thm:Cor 1.1 in Iksanov+Kabluchko:2016} and \ref{Thm:Gaussian boundary case},
we recall some basic facts about complex random variables.

\subsubsection*{Covariance calculations.}
The proofs of Theorems \ref{Thm:Cor 1.1 in Iksanov+Kabluchko:2016} and \ref{Thm:Gaussian boundary case}
are based on covariance calculations for complex random variables.
We remind the reader of some simple but useful facts in this context.
If $\zeta = \xi+\imag \tau$ is a complex random variable with $\xi = \Real(\zeta)$ and $\tau = \Imag(\zeta)$,
then a simple calculation shows that the covariance matrix of $\zeta$ can be represented as
\begin{equation}	\label{eq:complex covariance}
\Cov[\zeta]
= \begin{pmatrix}
\E[\xi^2]	&	\E[\xi\tau]	\\
\E[\xi\tau]	&	\E[\tau^2]
\end{pmatrix}
= \frac12 \begin{pmatrix}
\Real(\E[|\zeta|^2]+\E[\zeta^2])	&	\Imag(\E[|\zeta|^2]+\E[\zeta^2])	\\
\Imag(\E[|\zeta|^2]+\E[\zeta^2])	&	\Real(\E[|\zeta|^2]-\E[\zeta^2])
\end{pmatrix}.
\end{equation}
Thus covariance calculations can be reduced to second moment calculations.

\subsubsection*{Proof of Theorems \ref{Thm:Cor 1.1 in Iksanov+Kabluchko:2016} and \ref{Thm:Gaussian boundary case}.}
Throughout the paragraph,
for $n \in \N_0$ and $u \in \Gen_n$, we set $Y_u \defeq e^{-\lambda S(u)}/m^{n/2}$.
We start with a lemma.

\begin{lemma}	\label{Lem:L^2 norm of Z-1}
Suppose that $\lambda \in \Dom$ with $m(\lambda) \not = 0$ is such that $\sigma_\theta^2 < \infty$,
$\sigma_\lambda^2 > 0$ and $m(2\theta) < |m(\lambda)|^2$.
then
\begin{align}	\label{eq:L^2 norm of Z-1}
\E[|Z(\lambda)\!-\!1|^2] = \frac{\E[|Z_1\!-\!1|^2]}{1\!-\!\frac{m(2\theta)}{|m(\lambda)|^2}} < \infty
\text{ \ and \ }
\E[(Z(\lambda)\!-\!1)^2] = \frac{\E[(Z_1\!-\!1)^2] }{1\!-\!\frac{m(2\lambda)}{m(\lambda)^2}}.
\end{align}
\end{lemma}
\begin{proof}
Observe that \eqref{eq:gamma moment Z_1(theta)} and \eqref{eq:contraction condition} are satisfied with $\gamma=p=2$.
Consequently, $\E[|Z-1|^2] < \infty$.
In the next step, we calculate $\E[|Z-1|^2]$ and $\E[(Z-1)^2]$.
(Actually, the calculations below again give $\E[|Z-1|^2] < \infty$.)
As the increments of square-integrable martingales are uncorrelated,
\begin{align*}
\E[|Z-1|^2]
&= \lim_{n \to \infty} \E[(Z_n-1)(\overline{Z}_n-1)] = \sum_{n=0}^\infty \E[|Z_{n+1}-Z_{n}|^2]	\\
&= \E[|Z_1-1|^2] \sum_{n=0}^\infty \E\bigg[\sum_{|u|=n} \frac{e^{-2\theta S(u)}}{|m(\lambda)|^{2n}}\bigg]
= \frac{\E[|Z_1-1|^2]}{1-m(2\theta)/|m(\lambda)|^2}.
\end{align*}
Analogously, we infer
\begin{align*}
\E[(Z-1)^2]
&= \E[(Z_1-1)^2] \sum_{n=0}^\infty \E\bigg[\sum_{|u|=n} \frac{e^{-2\lambda S(u)}}{m(\lambda)^{2n}}\bigg]
= \frac{\E[(Z_1-1)^2] }{1-m(2\lambda)/m(\lambda)^2}.
\end{align*}
\end{proof}

Our combined proof of Theorems \ref{Thm:Cor 1.1 in Iksanov+Kabluchko:2016} and \ref{Thm:Gaussian boundary case} is based on an application of the Lindeberg-Feller central limit theorem.

\begin{proof}[Proof of Theorems \ref{Thm:Cor 1.1 in Iksanov+Kabluchko:2016} and \ref{Thm:Gaussian boundary case}]
Recall that $m = m(2\theta)$ if $|m(2\lambda)| < m(2\theta)$ and $m=m(2\lambda)$ if $|m(2\lambda)|=m(2\theta)$.
For $n \in \N$, define $c_n = 1$ in the situation of Theorem \ref{Thm:Cor 1.1 in Iksanov+Kabluchko:2016}
and $c_n \defeq n^{1/4}$ in the situation of Theorem \ref{Thm:Gaussian boundary case}.
Further, let $a_n \defeq c_n \frac{m(\lambda)^n}{m^{n/2}}$ for $n \in \N$.
Then \eqref{eq:decomposition nth gen} can be rewritten in the form
\begin{equation}	\label{eq:decomposition nth gen with Y'us}
a_n (Z-Z_n) = c_n \sum_{|u|=n} Y_u ([Z]_u-1).
\end{equation}
The right-hand side of \eqref{eq:decomposition nth gen with Y'us} given $\F_n$
is the sum of independent centered random variables.
We show that the distribution
of this sum given $\F_n$ converges in probability to the distribution
of a complex or real normal random variable.
To this end, we check the Lindeberg-Feller condition.
For any $\varepsilon > 0$, using that $|m|=m(2\theta)$, we obtain
\begin{align*}
\sum_{|u|=n} \! \E_n[|c_n Y_u ([Z]_u\!-\!1)|^2 \1_{\{|c_n Y_u ([Z]_u-1)|^2 > \varepsilon\}}]
&= c_n^2 \sum_{|u|=n} \! |Y_u|^2 \sigma_\lambda^2(\varepsilon c_n^{-2} |Y_u|^{-2})	\\
&= c_n^2 \sum_{|u|=n} \! \frac{e^{-2\theta S(u)}}{m(2 \theta)^{n}} \sigma_\lambda^2(\varepsilon c_n^{-2} |Y_u|^{-2})
\end{align*}
where, for $x \geq 0$,
\begin{equation*}
\sigma^2_\lambda(x) \defeq \E[|Z-1|^2 \1_{\{|Z-1|^2 > x\}}].
\end{equation*}
By Lemma \ref{Lem:L^2 norm of Z-1}, we have $\E[|Z-1|^2] < \infty$.
The dominated convergence theorem thus yields
$\sigma_\lambda^2(x) \downarrow 0$ as $x \uparrow \infty$.
Moreover, in the situation of Theorem \ref{Thm:Cor 1.1 in Iksanov+Kabluchko:2016},
\begin{equation*}
c_n^2 \sup_{|u|=n} |Y_u|^2 = \sup_{|u|=n} \frac{e^{-2\theta S(u)}}{m(2\theta)^{n}} \to 0	\quad	\text{a.\,s.\ as } n \to \infty
\end{equation*}
by \eqref{eq:Biggins:1998} (applied with $\theta$ replaced by $2\theta$).
In the situation of Theorem \ref{Thm:Gaussian boundary case},
\begin{equation*}
c_n^2 \sup_{|u|=n} |Y_u|^2 = n^{1/2} \sup_{|u|=n} e^{-V(u)} \to 0	\quad	\text{in } \Prob \text{-probability as } n \to \infty
\end{equation*}
by Proposition \ref{Prop:minimal position}.
In any case, we conclude that
\begin{align*}
\sum_{|u|=n} \! \E_n[|c_n Y_u ([Z]_u\!-\!1)|^2 \1_{\{|c_n Y_u ([Z]_u-1)|^2 > \varepsilon\}}] &	\\
\leq c_n^2 Z_n(2\theta) \sigma_\lambda^2(\varepsilon ({\textstyle c_n \sup_{|u|=n}} |Y_u|)^{-2}) & \to 0
\end{align*}
as $n \to \infty$ a.\,s.\ or in $\Prob$-probability, respectively, having utilized \eqref{eq:Aidekon+Shi} for the convergence in $\Prob$-probability.
By \eqref{eq:complex covariance}, covariance calculations can be reduced to calculations
for the second absolute (conditional) moment and the second (conditional) moment of $\frac{m(\lambda)^n}{m^{n/2}} (Z-Z_n)$:
\begin{align}
\E_n\big[|a_n (Z-Z_n)|^2\big]
&= c_n^2 \E_n\bigg[\bigg(\sum_{|u|=n} Y_u ([Z]_u-1)\bigg)\bigg(\sum_{|v|=n} \overline{Y}_v ([\overline{Z}]_v-1)\bigg)\bigg]	\notag	\\
&= c_n^2 \E_n\bigg[\sum_{|u|=n} |Y_u|^2 |[Z]_u-1|^2\bigg]
= \E[|Z-1|^2] c_n^2 \sum_{|u|=n} |Y_u|^2	\notag	\\
&= \E[|Z-1|^2] c_n^2 Z_n(2\theta).	\label{eq:second absolute conditional moment}
\end{align}
where the second equation follows from the fact that,
for $|u|=|v|=n$ with $u \not = v$,
$[Z]_u-1$ and $[Z]_v-1$ are independent and centered,
and hence the cross terms vanish.
The right-hand side of \eqref{eq:second absolute conditional moment}
converges to $\E[|Z-1|^2] Z(2\theta)$ a.\,s.\
in the situation of Theorem \ref{Thm:Cor 1.1 in Iksanov+Kabluchko:2016}
and to $\E[|Z-1|^2] (\frac{2}{\pi \sigma^2})^{1/2} D_\infty$ in $\Prob$-probability
in the situation of Theorem \ref{Thm:Gaussian boundary case}.
An analogous calculation gives
\begin{align}	\label{eq:second conditional moment}
\E_n\big[(a_n (Z-Z_n))^2]
= \E[(Z-1)^2] c_n^2 \sum_{|u|=n} Y_u^2.
\end{align}
We shall find the limit of the right-hand side of \eqref{eq:second conditional moment}, thereby verifying that the conditions \cite[Eqs.\;(2.5)--(2.7)]{Helland:1982} are fulfilled.
The claimed convergence then follows from the cited source
and the Cram\'er-Wold device \cite[p.\;87, Corollary 5.5]{Kallenberg:2002}.
In the situation of Theorem \ref{Thm:Cor 1.1 in Iksanov+Kabluchko:2016}, if $Z_n(2\theta) \to 0$ a.\,s., then $\sum_{|u|=n} Y_u^2\to 0$ a.\,s., so that nothing remains to be shown.
Thus, for the remainder of the proof, we suppose that $Z_n(2\theta)$ converges a.\,s.\ and in $L^1$ to $Z(2\theta)$ or that \eqref{eq:Aidekon+Shi} holds. We distinguish two cases.

\noindent
\emph{Case 1}: Let $|m(2\lambda)| < m(2\theta)$.
We apply Lemma \ref{Lem:cancellation} with $(L(u))_{u \in \Gen} = (Y_u^2)_{u \in \Gen}$.
In this case
\begin{equation*}
\E\bigg[\sum_{|u|=1} |L(u)|\bigg] = \E\bigg[\sum_{|u|=1} |Y_u|^2\bigg] = \E[Z_1(2\theta)] = 1.
\end{equation*}
Further,
\begin{equation*}
a \defeq \E\bigg[\sum_{|u|=1} L(u)\bigg] = \E\bigg[\sum_{|u|=1} Y_u^2\bigg] = \frac{m(2\lambda)}{m(2\theta)}
\end{equation*}
satisfies $|a|<1$.
When the assumptions of Theorem \ref{Thm:Cor 1.1 in Iksanov+Kabluchko:2016} hold, Lemma \ref{Lem:cancellation}(b) applies (with condition (i) satisfied)
and yields $\sum_{|u|=n} Y_u^2 \to 0$ in $\Prob$-probability. If, additionally, $2\theta \in \Lambda$, then
\begin{equation*}
\E\bigg[\sum_{|u|=1} |L(u)|^p\bigg] = \E\bigg[\sum_{|u|=1} |Y_u|^{2p} \bigg] = \frac{m(p2\theta)}{m(2\theta)^p} < 1
\end{equation*}
for some $p \in (1,2]$. Hence, $\sum_{|u|=n} Y_u^2 \to 0$ a.\,s.\ by Lemma \ref{Lem:cancellation}(a).
When the assumptions of Theorem \ref{Thm:Gaussian boundary case} hold,
we obtain $n^{1/2}\sum_{|u|=n} Y_u^2 \to 0$ in $\Prob$-probability by another appeal to Lemma \ref{Lem:cancellation}(b) (this time with condition (ii) satisfied).
Thus, under the assumptions of both theorems, the limit of the right-hand side of \eqref{eq:second conditional moment} vanishes.
\smallskip

\noindent
\emph{Case 2}: Let $|m(2\lambda)| = m(2\theta)$.
Then there exists some $\varphi \in [0,2\pi)$ such that $m(2\lambda) = m(2\theta) e^{\imag \varphi}$.
This implies $e^{-2 \imag \eta S(u)} = e^{\imag \varphi}$  for all $|u|=1$ a.\,s.,
equivalently, $S(u) \in \frac{-\varphi}{2\eta} + \frac{\pi}{\eta} \Z$ for all $|u|=1$ a.\,s.
Therefore, a.\,s.\ for every $u \in \Gen$,
\begin{align*}
e^{-\lambda S(u)} = e^{-\theta S(u)} e^{-\imag \eta S(u)} = \pm e^{\imag \varphi/2} e^{-\theta S(u)}
\end{align*}
and thereupon $m(\lambda) = e^{\imag \varphi/2} q$ where $q \in \R$ with $0<|q| \leq m(\theta)$.
Consequently, $Z_n(\lambda) \in \R$ a.\,s.\ for every $n \in \N_0$.
Thus, also $Z(\lambda) \in \R$ a.\,s.
Further $m(\lambda)^n/m^{n/2} = m(\lambda)^n/m(2\lambda)^{n/2} = q^n/m(2\theta)^{n/2} \in \R$.
Hence, all terms in \eqref{eq:second absolute conditional moment} and \eqref{eq:second conditional moment}
coincide and so do their limits.
\end{proof}

It is worth noting that in Case 2, in order to arrive at the stronger statement (weak convergence a.\,s.),
we do not need $2\theta \in \Lambda$,
but only require the uniform integrability of $(Z_n(2\theta))_{n \in \N_0}$ or equivalently \eqref{eq:nonzero}.

\section{The regime in which the extremal positions dominate}	\label{sec:extremal regime}

First recall that $V(u)$ is defined by \eqref{eq:V(u)}
and that $V_n(u)=V(u) - \frac32 \log n$ for $u \in \Gen_n$.
Further, for each $K\in\R$ define $f_K:\R \to [0,1]$ by
\begin{equation}	\label{eq:f_K}
f_K(x)	\defeq	\begin{cases}
				1		&	\text{for }	x \leq K,	\\
				K+1-x	&	\text{for }	K \leq x \leq K+1,	\\
				0		&	\text{for }	x \geq K+1.
				\end{cases}
\end{equation}
Our proof of Theorem \ref{Thm:domination of extremal particles}
is based on two lemmas about the processes $\mu_n$, $n \in \N$ and related point processes.
Proposition \ref{Prop:Madaule's theorem} tells us that
\begin{equation}	\label{eq:strength}	\textstyle
\int f \, \dmu_n \distto \int f \, \dmu_\infty	\quad	\text{as }	n \to \infty
\end{equation}
for all continuous and compactly supported $f:\R \to [0,\infty)$.
This taken together with information about the left tail of $\mu_n$ for large $n$ provided by \cite[Theorem 1.1]{Aidekon:2013} enables us to show
that relation \eqref{eq:strength} holds for a wider class of functions $f$.
This is the content of Lemma \ref{Lem:convergence of mu_n left-hand compactification}.

\begin{lemma}	\label{Lem:convergence of mu_n left-hand compactification}
Suppose that the assumptions of Theorem \ref{Thm:domination of extremal particles} are satisfied.
Then relation \eqref{eq:strength} holds for all continuous functions $f:\R \to [0,\infty)$ with $f(x) = 0$ for all sufficiently large $x$.
\end{lemma}
\begin{proof}
Pick an arbitrary continuous function $f:\R \to [0,\infty)$ satisfying $f(x) = 0$ for all sufficiently large $x$.
For any fixed $K\in\R$, the function $g_K(x) \defeq f(x) (1-f_K(x))$ is continuous and has a compact support.
Therefore, $\int g_K \dmu_n \distto \int g_K \dmu_\infty$ as $n \to \infty$ by Proposition \ref{Prop:Madaule's theorem}.
Since $\mu_\infty((-\infty,a]) < \infty$ a.\,s.\ for any $a\in\R$ by another appeal to Proposition \ref{Prop:Madaule's theorem}, we infer
\begin{equation}	\label{eq:remainder g_K}
\lim_{K \to -\infty} \int g_K \, \dmu_\infty = \int f \, \dmu_\infty	\quad	\text{a.\,s.}
\end{equation}
On the other hand, for any $\varepsilon > 0$,
\begin{align*}
\limsup_{n \to \infty}
\Prob\bigg(\bigg|\int f(x) f_K(x) \, \mu_n (\dx) \bigg| > \varepsilon\bigg)
&\leq \limsup_{n \to \infty} \Prob\big( \mu_n((-\infty, K+1])\geq 1\big)	\\
&= \limsup_{n \to \infty} \Prob\Big(\min_{|u|=n} V(u) - \tfrac32 \log n\leq K+1\Big)
\end{align*}
where $\min_\emptyset \defeq \infty$.
By \cite[Theorem 1.1]{Aidekon:2013},
\begin{equation*}
\lim_{K \to -\infty}\limsup_{n \to \infty} \Prob\Big(\min_{|u|=n} V(u) - \tfrac32 \log n\leq K+1\Big) = 0.
\end{equation*}
The latter limit relation, \eqref{eq:remainder g_K} and \cite[Theorem 4.2]{Billingsley:1968} imply
$\int f \dmu_n \distto \int f \dmu_\infty$.
\end{proof}

With Lemma \ref{Lem:convergence of mu_n left-hand compactification} at hand we can now show that for any $\gamma>1$
\begin{equation}	\label{eq:finite series}	\textstyle
\int e^{-\gamma x} \, \mu_{\infty}(\dx) = \sum_k e^{-\gamma P_k}<\infty	\quad	\text{a.\,s.}
\end{equation}
To see this, pick $M>0$ and consider the following chain of inequalities:
\begin{align*}
\Prob\bigg(\sum_j e^{-\gamma P_j} > M\bigg)
&= \sup_{K \in \N} \Prob\bigg(\sum_j e^{-\gamma P_j} f_K (P_j)> M\bigg)	\\
&\leq \sup_{K \in \N} \liminf_{n \to \infty} \Prob\bigg(\sum_{|u|=n} e^{-\gamma V_n(u)} f_K (V_n(u)) > M\bigg)	\\
&\leq \limsup_{n \to \infty} \Prob\bigg(\sum_{|u|=n} e^{-\gamma V_n(u)} > M\bigg)
\end{align*}
where Lemma \ref{Lem:convergence of mu_n left-hand compactification} and the Portmanteau theorem
have been used for the first inequality. The latter $\limsup$ tends to $0$ as $M \to \infty$ by \cite[Proposition 2.1]{Madaule:2017}.

Recall that $(Z^{(k)})_{k \in \N}$ denotes a sequence of independent copies of $Z(\lambda)-1$
which are also independent of $\mu_\infty = \sum_{k} \delta_{P_k}$.
We define point processes on $\R \times \C$ by
\begin{equation*}
\mu_\infty^* \defeq \sum_{k} \delta_{(P_k, Z^{(k)})}
\quad	\text{and}	\quad
\mu_n^* \defeq \sum_{|u|=n} \delta_{(V_n(u),[Z(\lambda)]_u-1)},	\quad	n \in \N.
\end{equation*}

\begin{lemma}	\label{Lem:convergence of mu_n^*}
Suppose that the assumptions of Theorem \ref{Thm:domination of extremal particles} are satisfied.
Then $\int f \, \dmu_n^* \distto \int f \, \dmu_\infty^*$ for all bounded continuous function $f:\R \times \C \to \C$
such that $f(x,z) = 0$ whenever $x$ is sufficiently large.
\end{lemma}
\begin{proof}
We derive the assertion from Lemma \ref{Lem:convergence of mu_n left-hand compactification}.
More precisely, first let $f:\R \times \C \to [0,\infty)$
be an arbitrary continuous function such that $f(x,z) = 0$ for all $z \in \C$ whenever $x$ is sufficiently large.
Since the convergence $\int f \, \dmu_n^* \distto \int f \, \dmu_\infty^*$ is equivalent to the convergence of the corresponding Laplace transforms
it suffices to show that the Laplace functional of $\mu_n^*$ at $f$
converges to the Laplace functional of $\mu_\infty^*$ at $f$.
To this end, define $\varphi(x) \defeq \E[\exp(-f(x,Z^{(1)}))]$ for $x\in\R$.
Clearly, $0 < \varphi \leq 1$. Further, the continuity of $f$ together with the dominated convergence theorem imply that $\varphi$ is continuous.
Therefore, $-\log \varphi:\R \to [0,\infty)$
is continuous. Since $f(x,z)=0$ for all sufficiently large $x$, the same is true for $- \log \varphi$.
Lemma \ref{Lem:convergence of mu_n left-hand compactification}
implies that $\int (-\log \varphi(x)) \, \mu_n(\dx) \distto \int (-\log \varphi(x)) \, \mu_\infty(\dx)$.
Using this, we find that the Laplace functional of $\mu_n^*$ evaluated at $f$ satisfies
\begin{align*}
\E \bigg[\exp\bigg(-\int f(x,y) \, \mu_n^*(\dx,\dy) \bigg) \bigg]
&= \E \bigg[ \E_n \bigg[ \exp\bigg(-\sum_{|u|=n} f(V_n(u),[Z(\lambda)]_u-1) \bigg) \bigg] \bigg]	\\
&= \E \bigg[\prod_{|u|=n} \varphi(V_n(u)) \bigg]	\\
&= \E \bigg[ \exp\bigg(-\!\sum_{|u|=n}\!(-\log \varphi(V_n(u))) \bigg)\bigg]	\\
&= \E \bigg[ \exp\bigg(-\!\int (-\log \varphi(x)) \, \mu_n(\dx)\bigg)\bigg]	\\
&\to \E \bigg[ \exp\bigg(-\!\int (-\log \varphi(x)) \, \mu_\infty(\dx)\bigg)\bigg]	\\
&= \E \bigg[\exp\bigg(-\int f(x,y) \, \mu_\infty^*(\dx,\dy) \bigg) \bigg].
\end{align*}
This completes the proof for nonnegative $f$.
For the general case, we decompose $f=f_1-f_2+\imag(f_3-f_4)$ with $f_{j}:\R \times \C \to [0,\infty)$ vanishing for large $x$.
Then for any nonnegative $\lambda_j$, from the first part, we conclude
\begin{equation*}
 \int (\lambda_1f_1+\lambda_2f_2+\lambda_3f_3+\lambda_4f_4) \, \dmu_n^*\distto \int (\lambda_1f_1+\lambda_2f_2+\lambda_3f_3+\lambda_4f_4) \, \dmu_\infty^*
\end{equation*}
and, in particular, we infer  $(\int f_j \, \dmu_n^*)_{j=1,\ldots,4} \distto (\int f_j \, \dmu_\infty^*)_{j=1,\ldots,4}$
from which we deduce the convergence  $\int f \, \dmu_n^* \distto \int f \, \dmu_\infty^*$.
\end{proof}

We now make the final preparations for the proof of Theorem \ref{Thm:domination of extremal particles}.
We have to show that $z_n (Z(\lambda)-Z_n(\lambda))$ converges in distribution
where
\begin{equation*}
z_n \defeq n^{\tfrac{3\lambda}{2 \vartheta}} \Big(\frac{m(\lambda)}{m(\vartheta)^{\lambda/\vartheta}}\Big)^n,	\quad	n \in \N.
\end{equation*}
We shall use the decomposition
\begin{align*}
z_n (Z(\lambda)-Z_n(\lambda))&
= \frac{z_n}{m(\lambda)^{n}} \sum_{|u|=n}e^{-\lambda S(u)}([Z(\lambda)]_u-1)		\\
&= \frac{z_n}{m(\lambda)^{n}} e^{n \frac{\lambda}{\vartheta} \log m(\vartheta)} \sum_{|u|=n}e^{-\frac{\lambda}{\vartheta} V(u)}([Z(\lambda)]_u-1)	\\
&=\sum_{|u|=n}e^{-\frac{\lambda}{\vartheta} V_n(u)}([Z(\lambda)]_u-1)	\\
&=\sum_{|u|=n}e^{-\frac{\lambda}{\vartheta} V_n(u)} f_K(V_n(u)) ([Z(\lambda)]_u-1)	\\
&\hphantom{=}+ \sum_{|u|=n}e^{-\frac{\lambda}{\vartheta} V_n(u)} \big((1 - f_K(V_n(u)))\cdot([Z(\lambda)]_u-1)\big)	\\
&\eqdef Y_{n,K}+R_{n,K}.
\end{align*}
We first check that the contribution of $R_{n,K}$ is negligible as $K$ tends to infinity.

\begin{lemma}	\label{Lem:remainder estimate}
If the assumptions of Theorem \ref{Thm:domination of extremal particles} hold,
then, for any $\delta>0$ and every measurable $h_K: \R \to [0,1]$ satisfying $0 \leq h_K \leq \1_{[K,\infty)}$
\begin{equation*}
\lim_{K\to\infty} \limsup_{n \to \infty} \Prob \bigg(\bigg|\sum_{|u|=n} \!\! e^{-\frac{\lambda}{\vartheta} V_n(u)} h_K(V_n(u)) ([Z(\lambda)]_u\!-\!1)\bigg| > \delta \bigg) = 0.
\end{equation*}
\end{lemma}
\begin{proof}
Let $\varepsilon, \delta > 0$
and
$1<\beta<\beta_0\defeq p \cdot \frac\theta\vartheta$.
From Proposition 2.1 in \cite{Madaule:2017}, we know that the sequence of distributions
of the random variables
$\sum_{|u|=n}e^{-\beta V_n(u)}$, $n \in \N$ is tight.
Therefore, there is an $M > 0$ such that 
$\sup_{n \in \N} \Prob(\mathcal Q_n) \leq \varepsilon$
where
\begin{equation*}
\mathcal Q_n
\defeq \Big\{\sum_{|u|=n}e^{-\beta V_n(u)}>M\Big\}.
\end{equation*}
Then
\begin{align*}
\Prob&\bigg(\bigg|\sum_{|u|=n} \!\! e^{-\frac{\lambda}{\vartheta} V_n(u)} h_K(V_n(u))([Z(\lambda)]_u-1) \bigg| > \delta\bigg)	\\
&\leq \Prob\bigg(\bigg|\sum_{|u|=n} \!\! e^{-\frac{\lambda}{\vartheta} V_n(u)} h_K(V_n(u))([Z(\lambda)]_u-1) \bigg| > \delta, \mathcal Q_n^\comp \bigg) + \varepsilon.
\end{align*}
We estimate the above probability using the following strategy. First, we use Markov's inequality for the function $x \mapsto |x|^p$.
Then, given $\F_n$, we apply Lemma \ref{Lem:TV-ineq complex}.
This gives
\begin{align}
\Prob&\bigg(\bigg|\sum_{|u|=n} \!\! e^{-\frac{\lambda}{\vartheta} V_n(u)} h_K(V_n(u)) ([Z(\lambda)]_u\!-\!1)\bigg| > \delta, \mathcal Q_n^\comp \bigg)	\notag	\\
&\leq  \frac4{\delta^{p}} \cdot \E[|Z(\lambda)\!-\!1|^p] \cdot \E\bigg[\sum_{|u|=n} \!\! e^{-\beta_0 V_n(u)} h_K(V_n(u))^p \1_{\mathcal Q_n^\comp} \bigg]	\notag	\\
&\leq  \frac4{\delta^{p}} \cdot \E[|Z(\lambda)\!-\!1|^p] \cdot e^{(\beta-\beta_0)K} \E\bigg[\sum_{|u|=n} \!\! e^{-\beta V_n(u)} \1_{\mathcal Q_n^\comp} \bigg]	\notag	\\
&\leq  \frac4{\delta^{p}} \cdot \E[|Z(\lambda)\!-\!1|^p] \cdot e^{(\beta-\beta_0)K} M.	\label{eq:bound for 1st prob}
\end{align}
The above bound does not depend on $n$ and, moreover,
tends to $0$ as $K \to \infty$. The latter is obvious  since $\beta < \beta_0$ and thus $\lim_{K \to \infty} e^{(\beta-\beta_0)K} = 0$.

\noindent
We conclude that
\begin{equation*}
\lim_{K\to\infty} \limsup_{n \to \infty} \Prob \bigg(\bigg|\sum_{|u|=n} \!\! e^{-\frac{\lambda}{\vartheta} V_n(u)} h_K(V_n(u)) ([Z(\lambda)]_u\!-\!1)\bigg| > \delta \bigg) \leq \varepsilon.
\end{equation*}
The assertion follows as we may choose $\varepsilon$ arbitrarily small.
\end{proof}

We are now ready to prove Theorem \ref{Thm:domination of extremal particles}.

\begin{proof}[Proof of Theorem \ref{Thm:domination of extremal particles}]
Define
\begin{equation*}
\hat Y_0\defeq 0\quad \text{and}\quad \hat Y_n \defeq \sum_{k=1}^n e^{-\frac{\lambda}{\vartheta}P_k^\ast}Z^{(k)}(\lambda),	\quad	n\in\N
\end{equation*}
and recall the notation $\beta_0= p \cdot \frac\theta\vartheta > 1$. Given $\mu_\infty$, for each $n\in\N$,
the random variable $\hat Y_n$ is the sum of complex-valued independent centered random variables.
An application of Lemma \ref{Lem:TV-ineq complex} yields
\begin{align*}
\E\big[|\hat Y_n|^p\big|\mu_\infty\big] &\leq 4\ \E[|Z(\lambda)\!-\!1|^p] \cdot \sum_{k=1}^n  e^{-\beta_0 P_k^\ast}
\leq 4\ \E[|Z(\lambda)\!-\!1|^p] \cdot \sum_{k}  e^{-\beta_0 P_k},
\end{align*}
and the latter sum is almost surely finite by \eqref{eq:finite series}. This shows that $(\hat Y_n)_{n\in\N_0}$, conditionally on $\mu_\infty$,
is an $L^p$-bounded martingale.
We conclude that $\hat Y_n$ converges a.\,s.\ conditionally on $\mu_\infty$, hence, also unconditionally thereby proving the first part of Theorem \ref{Thm:domination of extremal particles}.

The proof of the second part is based on an application of Theorem 4.2 in \cite{Billingsley:1968}
and the decomposition
\begin{align*}
z_n (Z(\lambda)-Z_n(\lambda)) = Y_{n,K}+R_{n,K}.
\end{align*}
In view of Lemma \ref{Lem:remainder estimate}, the cited theorem gives the assertion
once we have shown the following two assertions:
\begin{enumerate}[1.]
	\item	$Y_{n,K} \distto Y_K$ as $n \to \infty$ for every fixed $K>0$ where $Y_K$ is some finite random variable;
	\item	$Y_K \distto X_{\mathrm{ext}}$ as $K \to \infty$.
\end{enumerate}
The first assertion is a consequence of Lemma \ref{Lem:convergence of mu_n^*}.
Indeed, the function $(x,z) \mapsto  e^{-\frac{\lambda}{\vartheta} x} f_K(x) z$ is continuous and vanishes for all sufficiently large $x$.
Therefore, Lemma \ref{Lem:convergence of mu_n^*} yields
\begin{align*}
Y_{n,K} &= \sum_{|u|=n} e^{-\frac{\lambda}{\vartheta} V_n(u)} f_K(V_n(u)) ([Z(\lambda)]_u-1)	\\
&=		\int e^{-\frac{\lambda}{\vartheta} x} f_K(x) z \, \mu_n^*(\dx,\dz)
\distto	\int e^{-\frac{\lambda}{\vartheta} x} f_K(x) z \, \mu_\infty^*(\dx,\dz)
\eqdef	Y_K.
\end{align*}
Set $\hat Y\defeq \sum_k
e^{-\frac{\lambda}{\vartheta}P_k}Z^{(k)}(\lambda)$ and note that $\hat Y \eqdist X_{\mathrm{ext}}$.
To see that the second assertion holds,
we prove that $\E[|\hat Y-Y_K|^p|\mu_\infty]\to0$ a.\,s.\ as $K \to \infty$
which entails $Y_K\Probto \hat Y$ as $K\to\infty$.
To this end, we use (an infinite version of) Lemma \ref{Lem:TV-ineq complex} to obtain
\begin{align*}
 \E\big[|\hat Y-Y_K|^p|\mu_\infty]
 &\leq 4 \E[|Z(\lambda)-1|^p\big]\cdot \sum_{k}e^{-\beta_0 P_k}\big(1-f_K(P_k))^p\\
 &\leq 4 \E[|Z(\lambda)-1|^p]\cdot \sum_{k}e^{-\beta_0 P_k}\1_{\{P_k > K\}}.
\end{align*}
In view of \eqref{eq:finite series} the right-hand side converges to zero a.\,s.\ as $K\to\infty$. The proof of Theorem \ref{Thm:domination of extremal particles} is complete.
\end{proof}

\section{The boundary $\partial \Lambda^{(1,2)}$}

Throughout this section, we fix $\lambda \in \Dom$ and suppose that
\begin{equation*}	\tag{C1}	\textstyle
\frac{m(\alpha\theta)}{|m(\lambda)|^\alpha} = 1
\quad	\text{and}	\quad
\frac{\theta m'(\theta \alpha)}{|m(\lambda)|^\alpha} = \log(|m(\lambda)|)
\end{equation*}
holds with $\alpha \in (1,2)$,
i.e., $\lambda \in \partial \Lambda^{(1,2)}$ and that
there are $\gamma \in (\alpha,2]$ and $\kappa \in (\frac\alpha2,1)$ such that
\begin{equation*}
\text{\eqref{ass:Lgamma}}\quad	\E[Z_1(\theta)^\gamma] < \infty
\qquad	\text{and}	\qquad
\text{\eqref{ass:L2_locally}}\quad\E\big[Z_1(\kappa \theta)^2] < \infty.
\end{equation*}
As before, for $n \in \N_0$ and $u \in \Gen_n$, we set $L(u) \defeq e^{-\lambda S(u)}/m(\lambda)^n$,
and abbreviate $Z_n(\lambda)$ and $Z(\lambda)$ by $Z_n$ and $Z$, respectively.
Notice that $\sum_{|u|=n} |L(u)|^\alpha = \sum_{|u|=n} e^{-V(u)}$ for each $n \in \N_0$
where $V(u)$ is defined in \eqref{eq:V(u)},
i.e., $V(u)	\defeq \alpha \theta S(u) + |u| \log (m(\alpha \theta))$.
The assumptions of Theorem \ref{Thm:tail Z(lambda) in Lambda^(1,2)} guarantee that \eqref{eq:Aidekon+Shi} holds, that is,
\begin{equation}	\label{eq:asymptotics sum |L(u)|^alpha}
\sqrt{n} \sum_{|u|=n} |L(u)|^\alpha \Probto \sqrt{\frac{2}{\pi \sigma^2}} D_\infty.
\end{equation}
Indeed, \eqref{eq:C1} and \eqref{ass:Lgamma}
entail conditions \eqref{eq:normalized_V}, \eqref{eq:sigma^2<infty Aidekon+Shi} and \eqref{eq:Aidekon+Shi condition},
which are sufficient for \eqref{eq:Aidekon+Shi} to hold. To be more precise, \eqref{eq:C1} implies \eqref{eq:normalized_V}.
Further, the function
\begin{equation*}
\R \ni t \mapsto m_V(t) = \E\bigg[\sum_{|u|=1} e^{-tV(u)}\bigg] = \frac{m(\alpha \theta t)}{m(\alpha\theta)^t}
\end{equation*} is finite at $t=1/\alpha < 1$ since $\lambda \in \Dom$ satisfies \eqref{eq:C1}
and at $t=\gamma/\alpha>1$
since, by superadditivity,
\begin{equation*}
\bigg(\sum_{|u|=1} e^{-\theta S(u)}\bigg)^{\!\!\gamma} \geq \sum_{|u|=1} e^{-\gamma\theta S(u)},
\end{equation*}
and \eqref{ass:Lgamma} holds.
Therefore, $m_V$ is finite on $[1/\alpha,\gamma/\alpha]$ and analytic on $(1/\alpha,\gamma/\alpha)$.
In particular, the second derivative is finite at $t=1$, which yields \eqref{eq:sigma^2<infty Aidekon+Shi}.
Again by superadditivity, we conclude that
\begin{equation*}
\bigg(\sum_{|u|=1} e^{-\theta S(u)}\bigg)^{\!\!\gamma} \geq \bigg(\sum_{|u|=1} e^{-\alpha \theta S(u)}\bigg)^{\!\!\gamma/\alpha}.
\end{equation*}
Thus, \eqref{ass:Lgamma} implies the first condition in \eqref{eq:Aidekon+Shi condition}.
To see that the second condition in \eqref{eq:Aidekon+Shi condition} also holds,
pick $\delta > 0$ such that $\alpha-\delta > 1$ and use
\begin{equation*}
\bigg(\sum_{|u|=1} e^{-\theta S(u)}\bigg)^{\!\!\gamma} \geq \bigg(\sum_{|u|=1} e^{-(\alpha-\delta) \theta S(u)}\bigg)^{\!\!\frac\gamma{\alpha-\delta}}
\geq \delta^{\frac\gamma{\alpha-\delta}} \cdot \bigg(\sum_{|u|=1} e^{-\alpha \theta S(u)} (\theta S(u))_+ \bigg)^{\!\!\frac\gamma{\alpha-\delta}}.
\end{equation*}

\subsection{Martingale fluctuations on $\partial \Lambda^{(1,2)}$}

First, we show how from the knowledge of the tail behaviour of $Z(\lambda)$
we can deduce Theorem \ref{Thm:fluctuations on partial Lambda^(1,2)}.
To this end, suppose that the assumptions of Theorem \ref{Thm:tail Z(lambda) in Lambda^(1,2)} are satisfied.
Set
$W=Z-1$ and observe that $W$ has the same tail behavior as $Z$, i.e.,
\begin{equation}	\label{eq:tail of W}
\lim_{\substack{|z|\to 0, \\ z\in \gr}}\E[|z|^{-\alpha}\phi(zW)] = {\textstyle \int \phi \, \dnu}
\end{equation}
for any $\phi \in C_{\mathrm{c}}^{2}(\Chat \setminus \{0\})$.
To see that this is true, first notice that, for any $w \in \Chat \setminus \{0\}$ and  $z \in \gr$ such that $|z| \leq 1$, we have
\begin{equation*}
|\phi(w)-\phi(w-z)| = |z| \Big|\frac{\phi(w)-\phi(w-z)}{z} \Big| \leq |z| \sup_{u}|\nabla\phi(u)|\1_{P_1}(w),
\end{equation*}
where  $P_j$ is the $j$-neighborhood of $\supp \, \phi$, i.e.,  $P_j=\{u:|u-t| \leq j \mbox{ for some } t \in \supp \, \phi\}$.
Setting $\chi(w)\defeq \sup_{u}|\nabla \phi(u)| \1_{P_2}*\chi_0(w)$
where $\chi_0:\C \to [0,\infty)$ is a probability density function smooth on $\C$ and supported by the unit disc,
we infer that $\chi \in C_{\mathrm{c}}^{2}(\Chat \setminus \{0\})$ and
\begin{equation*}
|\phi(w)-\phi(w-z)| \leq |z| \chi(w).
\end{equation*}
Hence,
\begin{equation*}
\bigg|\lim_{\substack{|z|\to 0, \\ z\in \gr}} \E[|z|^{-\alpha}\phi(zW)] - \lim_{\substack{|z|\to 0, \\ z\in \gr}} \E[|z|^{-\alpha}\phi(zZ)]\bigg|
\leq \lim_{\substack{|z|\to 0, \\ z\in \gr}} |z| \cdot \E[|z|^{-\alpha} \chi(zZ)] = 0
\end{equation*}
where Theorem \ref{Thm:tail Z(lambda) in Lambda^(1,2)} has been used.

Our first result in this section is a corollary of Theorem \ref{Thm:tail Z(lambda) in Lambda^(1,2)}.
Recall that, for a complex number $z \in \C$, we sometimes write $z_1 = \Real(z)$ and $z_2 = \Imag(z)$.

\begin{corollary}	\label{Cor:Lambda^(1,2) expectation and covariance}
In the situation of Theorem \ref{Thm:tail Z(lambda) in Lambda^(1,2)}, for every $h>0$ with $\nu(\{y:|y|=h\})=0$
and every $j,k=1,2$,
we have
\begin{align}	\label{eq:Lambda^(1,2) covariance}
\lim_{\substack{|z|\to 0, \\ z\in \gr}} |z|^{-\alpha} \E[(zW)_j(zW)_k|;|zW| \leq h]	&=	\int_{\{|y|< h\}} y_j y_k \, \nu(\dy)	\\
\label{eq:Lambda^(1,2) expectation}
\text{and}	\quad
\lim_{\substack{|z|\to 0, \\ z\in \gr}} |z|^{-\alpha} \E[zW;|zW| \leq h]	&=	-\int_{\{|y|>h\}} y \, \nu(\dy).
\end{align}
\end{corollary}
\begin{proof}
We start with some preparations. Throughout the proof, when letting $|z|\to 0$ it is tacitly assumed that $z\in\gr$.
First, observe that
\begin{equation}	\label{eq:sup trunc 2nd moment}	\textstyle
\limsup_{|z| \to 0} |z|^{-\alpha} \E[|zW|^2; |zW|<\delta] \to 0	\quad	\text{as }\delta\to0.
\end{equation}
To see this, first choose a nonnegative function $\phi \in C_{\mathrm{c}}^{2}(\Chat \setminus \{0\})$
satisfying $\phi \geq \1_{\{|z| \geq 1\}}$.
Then, by \eqref{eq:tail of W},
\begin{equation*}	\textstyle
|z|^{-\alpha} \Prob(|zW|>1) \leq |z|^{-\alpha} \E[\phi(zW)] \to \int \phi \, \dnu
\end{equation*}
as $|z| \to 0$.
In particular, there is a finite constant $C > 0$ such that
\begin{equation}	\label{eq:bounded by C}
\sup_{0 < |z| \leq 1} |z|^{-\alpha} \Prob(|zW|>1) \leq C.
\end{equation}
In order to prove \eqref{eq:sup trunc 2nd moment}, pick $\delta \in (0,1)$.
We may suppose that $0 < |z| < \delta$.
Then
\begin{align*}
|z|^{-\alpha}\E[|zW|^2;|zW|<\delta]
&\leq |z|^{-\alpha}\E[(|zW|\wedge \delta)^2] \\
&= |z|^{-\alpha} \int_0^{|z|} 2t \Prob(|zW|>t) \, \dt +  |z|^{-\alpha} \int_{|z|}^{\delta} 2t \Prob(|zW|>t) \, \dt.
\end{align*}
The first integral can be bounded above by
\begin{align*}
|z|^{-\alpha} \int_0^{|z|} 2t \, \dt = |z|^{2-\alpha} \leq \delta^{2-\alpha}.
\end{align*}
Regarding the second integral, use \eqref{eq:bounded by C} to arrive at
\begin{align*}
|z|^{-\alpha} \int_{|z|}^{\delta} 2t \Prob(|zW|>t) \, \dt
= 2 \int_{|z|}^{\delta} t^{1-\alpha} \Big(\frac{|z|}{t}\Big)^{-\alpha}  \Prob\Big(\Big|\frac zt W\Big|>1\Big) \, \dt
\leq \frac{2C \delta^{2-\alpha}}{2-\alpha}.
\end{align*}
In conclusion, \eqref{eq:sup trunc 2nd moment} holds. Further, we have to show that
\begin{equation}	\label{eq:sup 1st moment tail}	\textstyle
\limsup_{|z| \to 0} |z|^{-\alpha} \E[|zW|; |zW| > K] \to 0	\quad	\text{as } K \to \infty.
\end{equation}
Indeed, in view of \eqref{eq:bounded by C}, we find
\begin{align*}
&\limsup_{|z| \to 0} |z|^{-\alpha} \E[|zW|; |zW| > K]	\\
&~= \limsup_{|z| \to 0}\bigg[\int_K^\infty t^{-\alpha} \Big|\frac{z}{t}\Big|^{-\alpha} \Prob\Big(\Big|\frac{z}{t}W\Big| > 1\Big) \, \dt 
+ K^{1-\alpha}\Big|\frac{z}{K}\Big|^{-\alpha}\Prob\Big(\Big|\frac{z}{K}W\Big| > 1\Big)\bigg]	\\
&~\leq CK^{1-\alpha}\Big(\frac{1}{\alpha-1}+1\Big)=\frac{CK^{1-\alpha}\alpha}{\alpha-1},
\end{align*}
which tends to zero as $K \to \infty$. Hence, \eqref{eq:sup 1st moment tail} holds.

We are ready to prove \eqref{eq:Lambda^(1,2) expectation}. To this end,
observe that $\E[zW;|zW| \leq h]=-\E[zW;|zW| > h]$ since $\E[W]=0$.
Now pick $0 < \delta < h < K$ such that $h+\delta < K$ and that
\begin{equation}\label{eq:cont}
\nu(\{y: |y|=h\}) = 0.
\end{equation}
Let $\phi\in C_{\mathrm{c}}^2(\Chat\setminus\{0\})$
be of the form $\phi(z) = z f(|z|)$ with twice continuously differentiable $f:[0,\infty) \to [0,1]$
satisfying $f(z)=0$ for $z\leq h$ and $f(z)=1$ for $z\in [h+\delta, K]$.
Then
\begin{align*}
&\limsup_{|z|\to 0}
\bigg|\int_{\{|y| > h\}} y \, \nu(\dy) - |z|^{-\alpha}\E[zW;  |zW| > h]\bigg|\\	
&~\leq \limsup_{|z|\to 0}
\bigg|\int \phi(y) \, \nu(\dy) - |z|^{-\alpha}\E[\phi(zW)]\bigg|	\\
&+ \int_{\{h < |y| < h+\delta\}} |y-\phi(y)| \, \nu(\dy) + \limsup_{|z| \to 0} |z|^{-\alpha}\E[|zW-\phi(zW)|; h<|zW|< h+\delta] 
\\
&~\hphantom{\leq}
+\int_{\{|y| > K\}} |y-\phi(y)| \, \nu(\dy) + \limsup_{|z| \to 0} |z|^{-\alpha} \E[|zW-\phi(zW)|; |zW| > K]\\
&~\leq
\int_{\{h < |y| < h+\delta\}} |y| \, \nu(\dy) + \limsup_{|z| \to 0} |z|^{-\alpha} (h+\delta) \Prob(h < |zW| < h+\delta)	\\
&~\hphantom{\leq}
+\int_{\{|y| > K\}} |y| \, \nu(\dy) + \limsup_{|z| \to 0} |z|^{-\alpha} \E[|zW|; |zW| > K]
\end{align*}
having utilized \eqref{eq:tail of W} and $|y-\phi(y)|\leq y$ for $|y| \in \C$.
The first (second) term on the right-hand side converges to zero as $\delta\to 0$ in view of \eqref{eq:cont}
(and suitable approximation of $\1_{\{h<|z|<h+\delta\}}$ by twice continuously differentiable functions with subsequent application of \eqref{eq:tail of W}).
The third and fourth term tend to $0$ as $K \to \infty$ by \eqref{eq:sup 1st moment tail}
and since $\int_{\{|y| \geq 1\}} |y| \, \nu(\dy) < \infty$.
The latter follows from the fact that $\nu(\{|y| \geq t\}) = t^{-\alpha} \nu(\{|y| \geq 1\})$
which is due to the $(\gr,\alpha)$-invariance of $\nu$.

Turning to the proof of \eqref{eq:Lambda^(1,2) covariance}, we fix $h>0$ satisfying \eqref{eq:cont} and pick $j,k \in \{1,2\}$.
For $0 < \delta < h/2$, choose $f \in C_{\mathrm{c}}^2((0,\infty))$ taking values in $[0,1]$
with $f = 0$ on $(0,\delta/2]$, $f = 1$ on $[\delta,h-\delta]$ and $f=0$ on $[h+\delta,\infty)$.
Define $\phi\in C_{\mathrm{c}}^2(\Chat\setminus\{0\})$ via $\phi(z) = z_j z_k f(|z|)$, $z \in \Chat$.
In particular, $\phi(z)=z_jz_k$ for $\delta \leq |z| \leq h-\delta$ and $\phi(z)=0$ for $|z|>h+\delta$.
Using \eqref{eq:tail of W} with this $\phi$ and \eqref{eq:sup trunc 2nd moment}
and arguing along the lines of the proof of \eqref{eq:Lambda^(1,2) expectation}, we conclude that
\begin{align*}
\limsup_{|z|\to 0}
& \bigg||z|^{-\alpha}\E[(zW)_j(zW)_k;|zW| \leq h]-\int_{\{|y| \leq h\}} y_j y_k \, \nu(\dy)\bigg|	\\
& \leq \limsup_{|z| \to 0} |z|^{-\alpha}\E[|zW|^2;|zW|<\delta]	\\
& \hphantom{\leq} +(h+\delta)^2 \limsup_{|z|\to 0} |z|^{-\alpha} \Prob(h-\delta<|zW|\leq h+\delta)	\\
& \hphantom{\leq} +\int_{\{|y| < \delta\}} |y|^2 \, \nu(\dy) + (h+\delta)^2 \nu(\{h-\delta<|y|\leq h+\delta\}).
\end{align*}
This bound tends to $0$ as $\delta \to 0$. We conclude that \eqref{eq:Lambda^(1,2) covariance} holds.
\end{proof}

We are now ready to prove Theorem \ref{Thm:fluctuations on partial Lambda^(1,2)}.

\begin{proof}[Proof of Theorem \ref{Thm:fluctuations on partial Lambda^(1,2)}]
For any strictly increasing sequence of natural numbers,
we can pass to a subsequence $(n_k)_{k \in \N}$
such that the convergence in \eqref{eq:asymptotics sum |L(u)|^alpha} and \eqref{eq:minimal position exponentiated} hold a.\,s.\ along this subsequence.
Once more, we use 
decomposition \eqref{eq:decomposition nth gen}.
First, we show that the triangular array $\{n_k^{w/(2\alpha)} L(u) ([Z]_u-1)\}_{|u|=n_k, k \in \N}$ is a null array.
Indeed,
\begin{equation*}
\sup_{|u|=n_k} \E_{n_k}[|n_k^{\frac{w}{2\alpha}} L(u)([Z]_u-1)| \wedge 1] \leq \E[|Z-1|] \cdot n_k^{\frac1{2\alpha}} \sup_{|u|=n_k} e^{-\frac1\alpha V(u)} \to 0 \text { a.\,s.}
\end{equation*}
as $k \to \infty$ by \eqref{eq:minimal position exponentiated}.
According to \cite[Theorem 15.28 and p.\;295]{Kallenberg:2002},
it suffices to prove that, for every $h>0$ with $\nu(\{z:|z|=h\})=0$,
\begin{align}
\label{eq:Kallenberg 15.28(i)}
\sum_{|u|=n_k} \!\!\! \Law(n_k^{\frac{w}{2\alpha}} L(u)[W]_u \mid \F_{n_k})	&\to	c D_{\infty}\nu	\text { vaguely in } \Chat \setminus \{0\}, \\
\label{eq:Kallenberg 15.28(ii)}
\sum_{|u|=n_k} \!\!\! \Cov_{n_k}[n_k^{\frac{w}{2\alpha}} L(u)[W]_u;|n_k^{\frac{1}{2\alpha}}L(u)[W]_u| \leq h]
&\to c D_{\infty} \!\!\!\! \underset{\{|z| \leq h\}}{\int} \!\!\!\! zz^\transp \, \nu(\dz)	\text{ a.\,s.},	\\
\label{eq:Kallenberg 15.28(iii)}
\sum_{|u|=n_k} \!\!\! \E_{n_k}[n_k^{\frac{w}{2\alpha}} L(u)[W]_u;|n_k^{\frac{1}{2\alpha}}L(u)[W]_u| \leq h]
&\to -c D_{\infty} \!\!\!\! \underset{\{|z| \leq h\}}{\int} \!\!\!\! z \, \nu(\dz) \text{ a.\,s.}
\end{align}
where $c=\sqrt{\frac{2}{\pi \sigma^2}}$.
Take any $\phi\in C_{\mathrm{c}}^2(\Chat \setminus \{0\})$.
Then, by \eqref{eq:tail of W},
\begin{align*}
\lim_{k \to \infty} \sum_{|u|=n_k} \E_{n_k}[\phi(n_k^{\frac{w}{2\alpha}}L(u)[W]_u)]
= \lim_{k \to \infty} n_k^{1/2} \sum_{|u|=n_k} |L(u)|^{\alpha} {\textstyle \int \phi \, \dnu}
=  c D_{\infty} {\textstyle \int \phi \, \dnu}
\end{align*}
a.\,s.\  proving \eqref{eq:Kallenberg 15.28(i)}.
Similarly, for \eqref{eq:Kallenberg 15.28(ii)} and \eqref{eq:Kallenberg 15.28(iii)},
we can apply \eqref{eq:Lambda^(1,2) covariance} and \eqref{eq:Lambda^(1,2) expectation}, respectively.
As a result we conclude that for any bounded continuous function $\psi: \C \to \R$
it holds that
\begin{equation}	\label{eq:convergence subsequence}
\E_{n_k}[\psi(n_k^{1/2\alpha}(Z-Z_{n_k}))] \to \E[\psi(X_{c D_{\infty}}) | \F_{\infty}]	\quad	\text{a.\,s.}
\end{equation}
with $c$ as before.
To summarize, we have shown that from any deterministic strictly increasing sequence of positive integers,
we can extract a deterministic subsequence $(n_k)_{k \in \N}$ such that \eqref{eq:convergence subsequence} holds.
In other words, for every bounded and continuous $\psi: \C \to \R$,
\begin{equation*}
\E_n[\psi(n^{\frac{w}{2\alpha}}(Z-Z_n))] \Probto \E[\psi(X_{c D_{\infty}}) | \F_{\infty}]	\quad	\text{as } n \to \infty,
\end{equation*}
i.e., \eqref{eq:Lambda^(1,2) weakly in probability} holds.
\end{proof}

\subsection{The tail behavior of $Z(\lambda)$ for $\lambda \in \partial \Lambda^{(1,2)}$}

\subsubsection*{An upper bound on the tails of the distribution of $Z(\lambda)$ for $\lambda \in \partial \Lambda^{(1,2)}$.}

\begin{proof}[Proof of Proposition \ref{Prop:heavy tail bound Z(lambda)}]
Proposition \ref{Prop:heavy tail bound Z(lambda)} can be proved along the lines of the proof of Theorem 2.1 in \cite{Kolesko+Meiners:2017}.
Equation (4.3) in the cited source carries over to the present situation, so it suffices to show that the truncated martingale
$(Z_n^{(t)}(\lambda))_{n \in \N_0}$ with increments
\begin{equation*}
Z_n^{(t)}-Z_{n-1}^{(t)} = \sum_{|u|=n-1} L(u) \1_{\{|L(u|_j)| \leq t \text{ for } j=0,\ldots,n-1\}}([Z_1]_u-1)
\end{equation*}
satisfies
\begin{equation*}
\sup_{n \in \N_0} \E[|Z_n^{(t)}-1|^p] \leq \text{const} \cdot t^{\gamma-\alpha}
\end{equation*}
where $1 \leq p < \alpha$ and the constant is independent of $t$. (This bound is analogous to (4.7) in \cite{Kolesko+Meiners:2017}.)
To prove the above uniform bound, one may argue as in the proof of \cite[Theorem 2.1]{Kolesko+Meiners:2017}
with $\phi(x) \defeq |x|^\gamma$.
What is more, the fact that this function is multiplicative and satisfies the assumptions of Lemma \ref{Lem:TV-ineq complex} (Topchi\u\i-Vatutin inequality for complex martingales),
allows for a substantial simplification of the proof given in \cite[Theorem 2.1]{Kolesko+Meiners:2017}.
A combination of the uniform moment bound above with formula (4.3) in \cite{Kolesko+Meiners:2017} yields the desired tail bound
$\Prob(|Z(\lambda)|>t) \leq \text{const} \cdot t^{-\alpha}$ for all $t>0$.
\end{proof}

\subsubsection*{Existence of the L\'evy measure $\nu$.}
We now prove the following, more detailed version of Theorem \ref{Thm:tail Z(lambda) in Lambda^(1,2)}.
The claim that the L\'{e}vy measure $\nu$ is non-zero which is not covered by Theorem \ref{Thm:tail Z(lambda) in Lambda^(1,2)} will be justified in the next subsection.
Recall that $\varrho$ denotes the Haar measure on $\gr$ normalized according to \eqref{eq:Haar normalization}.

\begin{theorem}	\label{Thm:tail Z(lambda) in Lambda^(1,2) with details}
Suppose that $\lambda \in \Dom$ and that the assumptions of Theorem \ref{Thm:tail Z(lambda) in Lambda^(1,2)} are satisfied.
Then there is a $(\gr,\alpha)$-invariant L\'evy measure $\nu$ on $\C \setminus \{0\}$ such that for any $\phi \in C_{\mathrm{c}}^{2}(\Chat \setminus \{0\})$,
we have
\begin{align}
\int \phi \, \dnu
&= \lim_{\substack{|z|\to 0, \\ z\in \gr}} |z|^{-\alpha}\E[\phi(zZ)]	\notag	\\
&= -\frac{2}{\sigma^2}\int |z|^{-\alpha} \log|z| \bigg(\E[\phi(zZ)]-\sum_{|u|=1}\E[\phi(zL(u)[Z]_u)]\bigg) \, \varrho(\dz)	\label{eq:int phi dnu}
\end{align}
where $\sigma^2=\E\big[\sum_{|u|=1}|L(u)|^{\alpha}(\log |L(u)|)^2\big]$.
Moreover,
\begin{equation}	\label{eq:kappa(alpha)=0}
\int |z|^{-\alpha} \bigg(\E[\phi(zZ)]-\sum_{|u|=1}\E[\phi(zL(u)[Z]_u)]\bigg) \, \varrho(\dz) = 0.
\end{equation}
\end{theorem}

For the proof of Theorem \ref{Thm:tail Z(lambda) in Lambda^(1,2) with details}, we need the following proposition.

\begin{proposition}	\label{Prop:renewal asymptotic}
Suppose that $(R_n)_{n \in \N_0}$ is a neighborhood recurrent multiplicative random walk on $\gr$
such that $\E[\log|R_1|]=0$ and $\sigma^2 \defeq \E[(\log|R_1|)^2] \in (0,\infty)$.
Further, suppose that $f,h: \gr \to \R$ are continuous functions satisfying
$|f(z)| \leq c_f (1 \wedge |z|^{-\delta})$ and $|h(z)| \leq c_h (|z|^{\delta} \wedge |z|^{-\delta})$
for some constants $c_f,c_h, \delta > 0$ and
\begin{equation*}
f(z)=\E[f(zR_1)]+h(z)	\quad	\text{for all } z \in \gr.
\end{equation*}
If there exist a sequence $(z_n)_{n \in \N}$ in $\gr$ with $z_n \to 0$ and a continuous function $g$ such that $f(z_n z) \to g(z)$ for all $z \in \gr$,
then
\begin{equation}	\label{eq:int h = 0}
\int h(z) \, \varrho(\dz) = 0
\quad	\text{and}	\quad
\lim_{\substack{|z|\to0,\\ z\in\gr}} f(z) = -\frac{2}{\sigma^2} \int h(z) \log |z| \, \varrho(\dz).
\end{equation}
\end{proposition}
\begin{proof}
From our assumptions and the dominated convergence theorem, we deduce
\begin{equation}	\label{eq:harmonic_function}
g(z) = \lim_{n \to \infty} \! f(zz_n) = \lim_{n \to \infty} \! (\E[f(zz_nR_1)] + h(zz_n)) = \lim_{n \to \infty} \! \E[f(zz_nR_1)] = \E[g(zR_1)].
\end{equation}
Consequently, $(g(zR_n))_{n \in \N_0}$ is a bounded martingale and, therefore, converges a.\,s.\ as $n \to \infty$.
On the other hand, $(R_n)_{n \in \N_0}$ is neighborhood recurrent on $\gr$.
Using the continuity of $g$, we conclude that $g$ is constant.

Now we define stopping times
$\tau \defeq \inf\{n \in \N: |R_n|<1\}$, $T_0 \defeq 0$ and, recursively,
$T_n=\inf\{k \geq T_{n-1}: |R_{T_k}|\ge|R_{T_{n-1}}|\}$ for $n \in \N$.
A variant of the duality lemma \cite[Lemma 4]{Kolesko:2013} then yields
\begin{align}	\label{eq:renewal equation}
\E\bigg[\int_{\{|R_{\tau}|\leq |z|<1\}}f(z_nz)\, \varrho(\dz)\bigg]=-\sum_{k=0}^{\infty}\E\bigg[\int_{\{|z|>|R_{T_k}z_n|\}} \!\! h(z)\, \varrho(\dz)\bigg].
\end{align}
The left-hand side converges to $g(1)\E[-\log |R_{\tau}|]$ as $n \to \infty$, and so does the right-hand side.
On the other hand, observe that the bound on $h$ implies that it is directly Riemann integrable (dRi) on $\gr$,
cf.\ \cite[p.\;396]{Buraczewski+al:2009} for the precise definition.
Moreover, for any $\rho>0$, the function $h_\rho(s) \defeq \1_{(\rho,\infty)}(|s|) \int_{\{|z|>|s|\}}h(z) \, \varrho(\dz)$ is also dRi on $\gr$.
Hence, we infer
\begin{align}
\sum_{k=0}^{\infty}&\E\bigg[\int_{\{|z|>|R_{T_k}z_n|\}} \!\! h(z)\, \varrho(\dz)\bigg]	\notag	\\
&= \sum_{k=0}^{\infty}\E\bigg[\1_{\{|R_{T_k}z_n| \leq \rho \}}\int_{\{|z|>|R_{T_k}z_n|\}} \!\! h(z)\, \varrho(\dz)\bigg] + \sum_{k=0}^{\infty}\E[h_\rho(R_{T_k}z_n)].	\label{eq:Green's function decomp}
\end{align}
Now suppose that $c \defeq \int h(z)\, \varrho(\dz) \not = 0$.
Then choose $\rho > 0$ so small that
\begin{equation*}
\bigg|\int_{\{|z|>|s|\}} h(z)\, \varrho(\dz)-c\bigg| < \frac{|c|}{2}
\end{equation*}
for all $|s| \leq \rho$.
From the renewal theorem for the group $\gr$ \cite[Theorem A.1]{Buraczewski+al:2009}, we conclude that
\begin{equation*}
\sum_{k=0}^{\infty}\E[h_\rho(R_{T_k}z_n)] \to \frac1{\E[\log|R_{T_1}|]} \int h_\rho(s) \varrho(\ds)	\quad	\text{as } n \to \infty.
\end{equation*}
On the other hand, the first infinite series in \eqref{eq:Green's function decomp} is unbounded as $n \to \infty$.
This is a contradiction and, hence, $\int h(z)\, \varrho(\dz) = 0$, which is the first equality in \eqref{eq:int h = 0}.
From \cite[Proposition 1]{Kolesko:2013} we infer that the function $s\mapsto\int_{\{|z|>|s|\}} h(z) \, \varrho(\dz)$ is also dRi with
\begin{equation*}
\iint_{\{|z|>|s|\}}h(z) \, \varrho(\dz) \, \varrho(\ds)=\int h(z) \log|z| \, \varrho(\dz).
\end{equation*}
An application of the  renewal theorem for the group $\gr$ \cite[Theorem A.1]{Buraczewski+al:2009} yields
that the right-hand side of \eqref{eq:renewal equation} converges to
\begin{equation*}
\frac{-1}{\E[\log|R_{T_1}|]} \int h(z) \log|z| \, \varrho(\dz).
\end{equation*}
Finally, since $\E[\log |R_{\tau}|] \cdot \E[\log|R_{T_1}|]=-\frac{\sigma^2}{2}$ by the proof of Theorem 18.1 on p.\;196 in \cite{Spitzer:1976},
we find
\begin{equation*}
g(1) = -\frac{2}{\sigma^2}\int h(z) \log|z| \, \varrho(\dz),
\end{equation*}
which does not depend on the sequence $(z_n)_{n \in \N}$.
\end{proof}

Further, we recall an elementary but useful fact.

\begin{proposition}	\label{Prop:Buraczewski+Damek+Mentemeier:2013 Lemma 6.2}
Let $\phi \in C_{\mathrm{c}}^{2}(\Chat \setminus \{0\})$.
Then, for any $0< \varepsilon \leq 1$, there is a finite constant $C > 0$ such that for any $n \in \N$ and any
$x_1,\dots,x_n \in \C$
it holds that
\begin{equation*}	\textstyle
\bigg|\phi(\sum_{k=1}^n x_k)-\sum_{k=1}^n\phi(x_k)\bigg|\leq C\sum_{1\leq j \neq k \leq n}|x_j|^{\varepsilon}|x_k|^{\varepsilon}.
\end{equation*}
\end{proposition}
\begin{proof}[Source]
The proposition which is almost identical with \cite[Lemma 6.2]{Buraczweski+Damek+Mentemeier:2013}
follows from the proof of the cited lemma.
\end{proof}

It is routine to check using induction on $n$ that the formula
\begin{equation}	\label{eq:Ef(R_n)}
\E[f(R_n)]=\E\bigg[\sum_{|u|=n}|L(u)|^{\alpha} f(L(u))\bigg],
\end{equation}
which is assumed to hold for any bounded and measurable function $f:\C \to \R$,
defines (the distribution of) a multiplicative random walk $(R_n)_{n \in \N_0}$ on $\gr$ with i.i.d.\ steps $R_n/R_{n-1}$, $n \in \N$.
From \eqref{eq:Ef(R_n)} for $n=1$ and \eqref{eq:C1}, we infer $\E[\log |R_1|]=0$, i.e., the random walk $(\log |R_n|)_{n\in\N_0}$ 
on $\R$ has centered steps and thus is recurrent. Consequently, $(R_n)_{n \in \N_0}$ is neighborhood recurrent.
Moreover, by \eqref{eq:Ef(R_n)} and \eqref{eq:sigma^2<infty Aidekon+Shi}, we have $\E[(\log|R_1|)^2] = \sigma^2 \in (0,\infty)$.

Theorem \ref{Thm:tail Z(lambda) in Lambda^(1,2) with details} will now be proved by an application of Proposition \ref{Prop:renewal asymptotic}.
\begin{proof}[Proof of Theorem \ref{Thm:tail Z(lambda) in Lambda^(1,2) with details}]
For any $z \in \gr$, we define a finite measure $\nu_z$ on the Borel sets of $\C$ via
\begin{equation*}
\nu_z(A) = |z|^{-\alpha} \Prob(z Z \in A).
\end{equation*}
First observe that, since $\Prob[|Z|>t] \leq Ct^{-\alpha}$ by Proposition \ref{Prop:heavy tail bound Z(lambda)},
the family of measures $\{\nu_z\}_{z\in\C\setminus\{0\}}$ as a subset of the set of locally finite measures on $\Chat\setminus\{0\}$
is relatively vaguely sequentially compact.
Let $(z_n)_{n \in \N}$ be a sequence in $\gr$ satisfying $z_n \to 0$ such that $\nu_{z_n}$ converges vaguely to some measure $\nu$.

Let $\phi \in C_{\mathrm{c}}^{2}(\Chat \setminus \{0\})$.
Define $f(z)=|z|^{-\alpha}\E[\phi(zZ)]$, $h(z)=f(z)-\E[f(zR_1)]$ for $z \in \C$.
We shall show that the assumptions of Proposition \ref{Prop:renewal asymptotic} are satisfied.
From the proposition we then infer that the limit $\lim_{n \to \infty} f(z_n)$ (hence, $\nu$) does not depend on the particular choice of $(z_n)_{n \in \N}$,
which implies that $\lim_{|z|\to 0, z\in \gr} f(z)$ exists.

First, notice that $f$ is continuous by the dominated convergence theorem and, thus,
also $h$ is continuous again by the dominated convergence theorem.
Since $\phi$ is bounded, we have $|f(z)| \leq \|\phi\|_\infty |z|^{-\alpha}$ where $\|\phi\|_\infty\defeq \sup_{x\in \Chat \setminus \{0\}}|\phi(x)|<\infty$.
Since, moreover, there is some $r>0$ such that $\phi(z) = 0$ for all $|z| \leq r$,
we infer $|f(z)| \leq |z|^{-\alpha} \|\phi\|_\infty \Prob(|zZ|>r) \leq C \|\phi\|_\infty r^{-\alpha}$ for all $|z|>0$.
Hence, $|f(z)| \leq c_f \cdot (1 \wedge |z|^{-\alpha})$ for all $z \not = 0$ and some $c_f > 0$.
Next, we show that
\begin{equation}\label{eq:ineq}
|h(z)| \leq c_h (|z|^{\delta}\wedge|z|^{-\delta})
\end{equation}
for some $c_h, \delta > 0$ and all $z \not =  0$.
In order to prove that $|h(z)| \leq c |z|^{-\delta}$ for some $c>0$, it suffices to give a corresponding bound on $|\E[f(zR_1)]|$.
To this end, we first notice that for any $s \in \R$, by \eqref{eq:Ef(R_n)}, we have
\begin{align*}
\E[|R_1|^s] = \E\bigg[\sum_{|u|=1} |L(u)|^\alpha |L(u)|^s \bigg] = \E\bigg[\sum_{|u|=1} \Big|\frac{e^{-\lambda S(u)}}{m(\lambda)}\Big|^{\alpha+s} \bigg]
= \frac{m((\alpha+s)\theta)}{|m(\lambda)|^{\alpha+s}}.
\end{align*}
Thus, $\E[|R_1|^s]<\infty$ iff $m((\alpha+s)\theta)<\infty$.
By assumption $m(\theta)<\infty$ and $m(\alpha \theta) < \infty$. Since $m$ is convex, it is finite on the whole interval $[\theta,\alpha\theta]$.
Therefore, for $\delta_1 \in (0,\alpha-1)$, we have $m((\alpha-\delta_1)\theta) < \infty$ and, equivalently,
$\E[|R_1|^{-\delta_1}]<\infty$.
Consequently,
\begin{equation*}
|\E[f(zR_1)]| \leq c_f \E[1 \wedge |zR_1|^{-\alpha}] \leq c_f \E[1 \wedge |zR_1|^{-\delta_1}] \leq c_f \E[|R_1|^{-\delta_1}] \cdot |z|^{-\delta_1}
\end{equation*}
for all $|z|>0$. It remains to show that we may choose $\delta > 0$ such that also $|h(z)| \leq C |z|^{\delta}$.
To this end, recall that $\kappa \in (\alpha/2,1)$ is such that $\E[Z_1(\kappa \theta)^2]<\infty$, see \eqref{ass:L2_locally}.
For this $\kappa$, we obtain
\begin{align*}
|h(z)|
&= |f(z)-\E[f(zR_1)]| = |z|^{-\alpha} \bigg|\E[\phi(zZ)]-\E\Big[\sum_{|u|=1}\phi(zL(u)[Z]_u)\Big]\bigg|	\\
&=
|z|^{-\alpha} \bigg|\E\bigg[\phi\bigg(\sum_{|u|=1} zL(u)[Z]_u\bigg) - \sum_{|u|=1}\phi( zL(u)[Z]_u)\bigg]\bigg|	\\
&\leq
C|z|^{2\kappa-\alpha}\E\bigg[\sum_{u \neq v}|L(u)[Z]_u|^{\kappa}|L(v)[Z]_v|^{\kappa}\bigg]
\end{align*}
by Proposition \ref{Prop:Buraczewski+Damek+Mentemeier:2013 Lemma 6.2}. The expectation in the above expression can be estimated by
\begin{align}	\label{eq:h_bound}
\E\bigg[\bigg(\sum_{|u|=1}|L(u)|^{\kappa}\bigg)^{\!\!2}\,\bigg] \, [\E[|Z|^{\kappa}]]^2  < \infty
\end{align}
where the finiteness is due to \eqref{ass:L2_locally}.
We have shown that \eqref{eq:ineq} holds with $\delta=\delta_1$ for any $\delta_1\in (0, (2\kappa-\alpha)\wedge (\alpha-1))$.

Since $\nu_{z_n}$ converges vaguely to $\nu$, we have, for any $\phi\in C_{\mathrm{c}}^{2}(\Chat \setminus \{0\})$,
\begin{equation}\label{eq:inv1}
\lim_{n\to\infty} f(z_n) = \lim_{n\to\infty}	\textstyle	\int \phi\, \dnu_{z_n} = \int \phi \, \dnu.
\end{equation}
Fix any $z\in \gr$. Then
\begin{equation}	\label{eq:inv2}
\lim_{n\to\infty}f(z_nz) = |z|^{-\alpha}\lim_{n\to\infty}	\textstyle	\int \phi(zx)\,\nu_{z_n}(\dx)=|z|^{-\alpha}\int \phi(zx)\, \nu(\dx)=: g(z)
\end{equation}
because the function $t\mapsto \phi(tz)$ still belongs to $C_{\mathrm{c}}^{2}(\Chat \setminus \{0\})$. Finally, we observe that the function $t\mapsto g(t)$ is continuous on $\Chat \setminus \{0\}$.
Consequently, Proposition \ref{Prop:renewal asymptotic} applies
and shows that \eqref{eq:kappa(alpha)=0} holds and that
\begin{align*}
\lim_{\substack{|z|\to 0,\\ z\in\gr}} |z|^{-\alpha} & \E[\phi(zZ)]
= -\frac{2}{\sigma^2} \int h(z) \log |z| \, \varrho(\dz)	\\
&= -\frac{2}{\sigma^2} \int \big(f(z)\!-\!\E[f(zR_1)]\big) \log |z| \, \varrho(\dz)	\\
&= -\frac{2}{\sigma^2} \int |z|^{-\alpha} \log |z| \bigg(\E[\phi(zZ)]\!-\! \E\bigg[\sum_{|u|=1} \phi(zL(u) [Z]_u)\bigg]\bigg)\, \varrho(\dz).
\end{align*}
This proves \eqref{eq:int phi dnu}. The $(\gr, \alpha)$-invariance of $\nu$ follows from the fact that the right-hand sides of \eqref{eq:inv1} and \eqref{eq:inv2} are equal for each $z\in\gr$.
\end{proof}

\subsubsection*{The measure $\nu$ is non-zero.}
In order to show that the L\'evy measure $\nu$ is non-zero,
we adopt the analytic argument invented in \cite{Buraczewski+al:2009}.
To this end, we take a nondecreasing
function $\varphi\in C^2(\Rp)$ such that $\varphi(t)=0$ for $t<1/2$ and $\varphi(t)=1$ for $t>1$.
We further set $\varphi(0) \defeq 0$ and $\varphi(z) \defeq \varphi(|z|)$ if $z \in \C \setminus \{0\}$.
For $s \in \C$, define  $\kappa(s)$ by
\begin{align*}
\kappa(s)
&\defeq \int |z|^{-s} \bigg(\E\bigg[\varphi\bigg(\sum_{|u|=1}zL(u)[Z]_u\bigg)-\sum_{|u|=1}\varphi(zL(u)[Z]_u)\bigg]\bigg)\, \varrho(\dz)
\end{align*}
whenever the absolute value of the integrand is $\varrho$-integrable.
Using the linearity of the $\varrho$-integral, we may write this integral
as the difference of two $\varrho$-integrals.
Straightforward estimates now show that both these integrals are finite if $\Real(s) \in (1,\alpha)$
since the latter entails $m(\Real(s) \theta) < \infty$.
For $s$ in the strip $1 < \Real(s) < \alpha$, using Fubini's theorem and the invariance of the Haar measure $\varrho$, we may rewrite $\kappa(s)$ in the form
\begin{align}	\label{eq:kappa FE}
\kappa(s)
&=\int |z|^{-s}\varphi(z) \, \varrho(\dz)\bigg(\E\bigg[ |Z|^{s}-\sum_{|u|=1}|L(u)[Z]_u|^{s}\bigg]\bigg)	\notag	\\
&=\int |z|^{-s}\varphi(z) \, \varrho(\dz)\cdot \Big(1-\frac{m(s \theta)}{|m(\lambda)|^s}\Big) \cdot \E[|Z|^{s}].
\end{align}
\begin{lemma}	\label{Lem:Goldie++}
For any $s > 1$ there exists a finite constant $C_s$ such that, for any $x,y \in \C$, we have
\begin{equation*}
\int |z|^{-s}|\varphi(zy)-\varphi(zx)|\, \varrho(\dz) \leq C_s ||y|^{s}-|x|^{s}|
\end{equation*}
and $\sup_{s \in I} C_s < \infty$ for every closed interval $I \subset (1,\infty)$.
\end{lemma}
\begin{proof}
Without loss of generality we assume that $|x|\leq |y|$.
Then $|\varphi(zy)-\varphi(zx)|$ may only be positive when $|zy|>1/2$ and $|zx|<1$.
We shall consider the three cases $|zx|<1/2<|zy|<1$, $|zx|<1<|zy|$ and $1/2<|zx|<|zy|<1$ separately. In the second case, we conclude that the integral in focus does not exceed
\begin{align*}
\int_{\{|x|<\frac1{|z|}<|y|\}}|z|^{-s}|\varphi(zy)-\varphi(zx)|\, \varrho(\dz)
&\leq	\int_{\{|x|<\frac1{|z|}<|y|\}}|z|^{-s}\, \varrho(\dz)
=	\frac1s (|y|^{s}-|x|^{s}).
\end{align*}
Analogously, in the first case, we obtain the upper bound
\begin{align*}
\int_{\{|x|<\frac1{2|z|}<|y|\}}|z|^{-s}|\varphi(zy)-\varphi(zx)|\, \varrho(\dz)
\leq	\frac1{s 2^s} (|y|^{s}-|x|^{s}).
\end{align*}
It remains to get the bound in the third case.
Since the derivative of $\varphi$ (as a function on $\Rp$) is bounded so that $\|\varphi'\|_\infty=\sup_{x>0}\varphi'(x)<\infty$ we have the upper bound
\begin{multline*}
\underset{\{|y|<\frac{1}{|z|}<2|x|\}}{\int} \!\!\!\!\!\!\!\!\! |z|^{-s}|\varphi(zx)-\varphi(zy)|\, \varrho(\dz)
\leq		\|\varphi'\|_{\infty} \!\!\! \underset{\{|y|<\frac{1}{|z|}<2|x|\}}{\int} \!\!\!\!\!\!\!\!\! |z|^{-s}(|zy|-|zx|)\, \varrho(\dz)	\\
=		\frac{\|\varphi'\|_{\infty}}{s-1} (|y|-|x|)(|2x|^{s-1}-|y|^{s-1})
\leq 	\frac{\|\varphi'\|_{\infty}}{s-1} 2^{s-1} (|y|^s-|x|^s).
\end{multline*}
The latter follows from the elementary inequality
\begin{equation*}
(t-1)(2^{s-1}-t^{s-1})\leq2^{s-1}(t-1)\leq2^{s-1}(t^{s}-1),
\end{equation*}
that we use for $t=\frac{|y|}{|x|}\in(1,2)$.
The claim concerning the local boundedness of $s\mapsto C_s$ is now obvious.
\end{proof}

\begin{lemma}	\label{Lem:analytic continuation of kappa}
Suppose that \eqref{ass:Lgamma} and \eqref{ass:L2_locally} hold.
Then $\kappa$ is well-defined and holomorphic on the strip $1<\Real(s)<\alpha+\epsilon$ for some $\epsilon>0$.
\end{lemma}
\begin{proof}
Let $s \in (1,\gamma)$ where $\gamma \in (\alpha,2]$ is as in \eqref{ass:Lgamma}. First, we show that
\begin{align}
&\E\bigg[ \bigg|\bigg|\sum_{|u|=1}L(u)[Z]_u\bigg|^{s}-\max_{|u|=1}|L(u)[Z]_u|^{s}\bigg|\bigg]<\infty,	\label{eq:10}	\\
&\E\bigg[\sum_{|u|=1}|L(u)[Z]_u|^{s}-\max_{|u|=1} |L(u)[Z]_u|^{s}\bigg]<\infty,	\quad	  \text{and}	\label{eq:11}	\\
&\E\bigg[\int |z|^{-s}\bigg(\sum_{|u|=1} h(|zL(u)[Z]_u|)-h({\textstyle \max_{|u|=1}|zL(u)[Z]_u|})\bigg)\, \varrho(\dz)\bigg]<\infty	\label{eq:12}
\end{align}
where $h=\1_{(1,\infty)}$ or $h=\varphi$ for $\varphi$ defined in the paragraph preceding Lemma \ref{Lem:Goldie++}.
In order to prove \eqref{eq:10}, we first observe that
\begin{align*}
&\E\bigg[\bigg|\bigg|\sum_{|u|=1} L(u)[Z]_u \bigg|^{s}-\bigg(\sum_{|u|=1}|L(u)[Z]_u|^2\bigg)^{\!\!s/2}\bigg|\bigg]	\\
&=\E\bigg[\bigg|\bigg(\sum_{|u|=1} \! |L(u)[Z]_u|^2+\!\!\!\! \sum_{\substack{|u|=|v|=1\\ u\neq v}} \!\!\!\!\! \Real(L(u)[Z]_u\overline{L(v)[Z]_v)}\bigg)^{\!\!s/2} \!\!\!\! -\bigg(\sum_{|u|=1} \! |L(u)[Z]_u|^2\bigg)^{\!\!s/2}\bigg|\bigg]	\\
&\leq \E\bigg[\bigg(\sum_{\substack{|u|=|v|=1\\ u\neq v}} |L(u)[Z]_uL(v)[Z]_v|\bigg)^{\!\!s/2}\bigg]	\\
&= \E\bigg[\E_1\bigg[\bigg(\sum_{\substack{|u|=|v|=1\\ u\neq v}} |L(u)[Z]_uL(v)[Z]_v|\bigg)^{\!\!s/2}\bigg]\bigg]  \\
&\leq (\E[|Z|])^s \cdot \E\bigg[\bigg(\sum_{\substack{|u|=|v|=1\\ u\neq v}}|L(u)||L(v)|\bigg)^{\!\!s/2}\bigg]	\\
&\leq (\E[|Z|])^s \cdot \E\bigg[\bigg(\sum_{|u|=1}|L(u)|\bigg)^{\!\!s}\bigg]<\infty,
\end{align*}
where we used the subadditivity on $[0,\infty)$ of the function $t \mapsto t^{s/2}$ for the first inequality,
Jensen's inequality for the second, and \eqref{ass:Lgamma} to conclude the finiteness.
This in combination with the inequality
\begin{align*}
0 \leq \bigg(\sum_{|u|=1}|L(u)[Z]_u|^2\bigg)^{\!\!s/2}\!\!\!-\max_{|u|=1} |L(u)[Z]_u|^{s}\leq \sum_{|u|=1}|L(u)[Z]_u|^{s}-\max_{|u|=1} |L(u)[Z]_u|^{s},
\end{align*}
which follows from the aforementioned subadditivity, shows that \eqref{eq:10} is a consequence of \eqref{eq:11}.

For $h=\1_{(1,\infty)}$ or $h=\varphi$ and positive $x_j$, we have
\begin{align*}
(1-h(\max_j x_j))
&\leq		\prod_{j}(1-h(x_j)) +\bigg(\sum_{j \neq k} \1_{\{x_j>1/2,\ x_k>1/2\}}\bigg) \wedge 1.
\end{align*}
Further, there exists some finite $C \geq 1$ such that $x-1+e^{-x} \leq C (x \wedge x^2)$ for all $x \geq 0$
and $\Prob(|Z|>t) \leq C t^{-\alpha}$, $\Prob(|Z|>t) \leq C t^{-1}$ for all $t>0$ (the latter follows from Markov's inequality).
Using these facts, we infer
\begin{align}
\int & |z|^{-s} \bigg(\E_1 \bigg[\sum_{|u|=1} h(|zL(u)[Z]_u|)-h(\max_{|u|=1}|zL(u)[Z]_u|)\bigg]\bigg)\, \varrho(\dz)	\notag	\\
&\leq \int |z|^{-s}\bigg(\sum_{|u|=1} \E_1[h(|zL(u)[Z]_u|)]-1+e^{-\sum_{|u|=1}\E_1[h(|zL(u)[Z]_u|)]}\bigg)\, \varrho(\dz)	\notag	\\
&\hphantom{\leq}\
+\int |z|^{-s} \bigg(\bigg(\sum_{\substack{|u|=|v|=1\\ u\neq v}}\Prob_1(|zL(u)[Z]_u|>1/2,\ |zL(v)[Z]_v|>1/2)\bigg)\wedge 1 \bigg) \, \varrho(\dz)	\notag	\\
&\leq C\int |z|^{-s}\bigg(\sum_{|u|=1}\E_1[h(|zL(u)[Z]_u|)]\bigg)\wedge\bigg(\sum_{|u|=1}\E_1[h(|zL(u)[Z]_u|)]\bigg)^{\!\!2} \, \varrho(\dz)	\notag	\\
&\hphantom{\leq}\
+4^\alpha C^2 \int |z|^{-s}\bigg(\bigg(\sum_{\substack{|u|=|v|=1\\ u\neq v}}|zL(u)|^{\alpha}|zL(v)|^{\alpha}\bigg)\wedge1 \bigg)\, \varrho(\dz)	\notag	\\
&\leq 4 C^3 \int |z|^{-s}\bigg(\sum_{|u|=1}|L(u)| |z| \bigg)\wedge\bigg(\sum_{|u|=1}|L(u)||z|\bigg)^{\!\!2}\,\varrho(\dz)	\notag	\\
&\hphantom{\leq}\
+4^\alpha C^2\bigg(\sum_{\substack{|u|=|v|=1\\ u\neq v}}|L(u)|^{\alpha}|L(v)|^{\alpha}\bigg)^{\!\!s/2\alpha}\int |z|^{-s}(|z|^{2\alpha}\wedge 1)\,
\varrho(\dz)	\notag \\
&\leq \bigg(\!4 C^3 \! \int |z|^{-s} (|z| \wedge |z|^2) \,\varrho(\dz) + 4^\alpha C^2 \! \int |z|^{-s}(|z|^{2\alpha}\wedge 1)\, \varrho(\dz)\!\bigg) \notag	\\
&\hphantom{\leq}~\cdot\bigg(\sum_{|u=1|} |L(u)| \bigg)^{\!\!s}
\label{eq:9}
\end{align}
where in the last step we have used that $(\sum_{|u|=1}|L(u)|^{\alpha})^{s/\alpha} \leq (\sum_{|u|=1}|L(u)|)^{s}$.
Notice that the last two $\varrho$-integrals in \eqref{eq:9} are finite since $1<s<2$.
Assumption \eqref{ass:Lgamma} entails $\E[(\sum_{|u|=1}|L(u)|)^{s}] <\infty$ which proves \eqref{eq:12}.
Choosing $h=\1_{(1,\infty)}$ and taking the expectation in \eqref{eq:9}, we conclude that \eqref{eq:11} holds.
In particular, for $s\in\C$ with $1<s_1=\Real(s)<\gamma$, we infer
\begin{align*}
 \int& \bigg||z|^{-s} \bigg(\E[\varphi(zZ)]-\sum_{|u|=1}\E[\varphi(zL(u)[Z]_u)]\bigg)\bigg|\, \varrho(\dz)\\
 &=\int |z|^{-s_1} \bigg|\E[\varphi(zZ)]-\sum_{|u|=1}\E[\varphi(z L(u)[Z]_u)]\bigg|\, \varrho(\dz)\\
 &\leq \int |z|^{-s_1} \bigg|\E\bigg[\varphi\bigg(z \sum_{|u|=1}L(u)[Z]_u\bigg)\bigg]-\E\Big[\varphi(\max_{|u|=1}|z L(u)[Z]_u|)\Big]\bigg|\, \varrho(\dz)\\
 &\quad+ \int |z|^{-s_1} \bigg(\sum_{|u|=1}\E[\varphi(z L(u)[Z]_u)]-\E\Big[\varphi(\max_{|u|=1}|z L(u)[Z]_u|)\Big]\bigg)\, \varrho(\dz)\\
 &\leq C_{s_1} \cdot \E\bigg[\bigg|\bigg|\sum_{|u|=1}L(u)[Z]_u\bigg|^{s_1}-\max_{|u|=1}|L(u)[Z]_u|^{s_1}\bigg|\bigg]	\\
 &\quad+\int |z|^{-s_1} \bigg(\sum_{|u|=1}\E[\varphi(z L(u)[Z]_u)]-\E\Big[\varphi(\max_{|u|=1} |z L(u)[Z]_u|)\Big]\bigg)\, \varrho(\dz)<\infty
\end{align*}
by Lemma \ref{Lem:Goldie++}, \eqref{eq:10} and \eqref{eq:12}. Therefore, $\kappa$ is well defined on the strip $1<\Real(s) < \gamma$.
Moreover, for any closed triangle $\Delta$ in $1<\Real(s)<\gamma$, by the above calculation,
we can apply Fubini's theorem and the holomorphy of $s \mapsto |z|^{-s}$ in $\Real(s)>1$ to conclude that
\begin{align*}
&\int_{\partial \Delta} \kappa(s) \, \ds	\\
&= \iint_{\partial \Delta}|z|^{s} \, \ds \bigg(\E\bigg[\varphi\bigg(\bigg|\sum_{|u|=1}zL(u)[Z]_u\bigg|\bigg)-\sum_{|u|=1}\varphi(|zL(u)[Z]_u|)\bigg]\bigg)\, \varrho(\dz)=0,
\end{align*}
which implies that $\kappa$ is holomorphic on the strip $1<\Real(s)<\gamma$ by Morera's theorem.
\end{proof}

\begin{theorem}	\label{Thm:positivity of nu}
If \eqref{ass:Lgamma} and \eqref{ass:L2_locally} hold, then the L\'evy measure $\nu$ is non-zero.
\end{theorem}
\begin{proof}
For $s\in\C$ with $0 < \Real(s)<\alpha$, we define the holomorphic function $F(s)\defeq \E[|Z|^s]$.
The functions $\kappa$ and $F$ are related by the identity, valid for $1 < \Real(s)<\alpha$,
\begin{align}	\label{eq:F(s)}
F(s) = \frac{\kappa(s)}{1-\frac{m(s \theta)}{|m(\lambda)|^s}} \cdot \bigg(\int |z|^{-s}\varphi(|z|) \, \varrho(\dz)\bigg)^{-1},
\end{align}
which is a direct consequence of \eqref{eq:kappa FE}.
According to Lemma \ref{Lem:analytic continuation of kappa}, $\kappa$ possesses a holomorphic extension to some neighborhood of $\alpha$.
Assuming that $\nu = 0$ we show that $F$ has such an extension as well. The latter statement will lead to a contradiction.
Indeed, if $\nu=0$, then $\int \varphi(|z|) \, \nu(\dz) = 0$, which together with \eqref{eq:int phi dnu} shows that $\kappa'(\alpha)=0$.
Since also $\kappa(\alpha)=0$ by \eqref{eq:kappa(alpha)=0},
we infer that the numerator in \eqref{eq:F(s)} has a $0$ of at least second order at $\alpha$, while
the denominator has a zero of at most second order at $\alpha$ by virtue of
\begin{equation*}
\frac{\mathrm{d}^2}{\ds^2} \frac{m(s\theta)}{|m(\lambda)|^s}
= \E \bigg[\sum_{|u|=1} \log^2 \Big(\frac{e^{-\theta S(u)}}{|m(\lambda)|}\Big) \Big(\frac{e^{-\theta S(u)}}{|m(\lambda)|}\Big)^s \bigg] > 0
\end{equation*}
for all $1 < s < \alpha+\epsilon$, in particular for $s=\alpha$.
Hence, $F$ does possess a holomorphic extension to some neighborhood of $\alpha$.
We conclude from Landau's theorem (cf. \cite[Theorems 5a and 5b in Chap. II]{Widder:1946}) that
\begin{equation}	\label{eq:higher moment}
\E[|Z|^{\alpha+\delta}]<\infty,
\end{equation}
for some $\delta>0$.

By $\I_n$ let us denote an increasing family of subtrees of the Harris-Ulam tree $\I$ such that $\I_1=\{\varnothing\}$,
$|\I_n|=n$ and $\bigcup_{n\in\N} \I_n=\I$.
By a classical diagonal argument such a family exists.
Write $u_n$ for the unique vertex from the set $\I_n\setminus \I_{n-1}$, $n \in \N$.
Next, we define
\begin{align*}
M_n \defeq \sum_{k=1}^n \! L(u_k)([Z_1]_{u_k}\!-\!1)
\text{ and }
M_n^{(t)}=\sum_{k=1}^{n} \! L(u_k)\1_{\{L(u_k|_j) \leq t\,\text{for}\,\text{all}\,j \leq |u_k|\}}([Z_1]_{u_k}\!-\!1),
\end{align*}
and observe that they constitute martingales with respect to the filtration $(\H_n)_{n \in \N_0}$
where $\H_n \defeq \sigma(\cZ(u_k): k=1,\ldots,n)$ for $n \in \N_0$.
We claim that $M_n$ converges a.\,s.\ to $Z-1$.
To see this, recall from the proof of \cite[Theorem 2.1]{Kolesko+Meiners:2017} that on the set $\{\max_{v \in \Gen} |L(v)| \leq t\}$,
we have
\begin{equation*}
Z_n-1 = Z_n^{(t)}-1=\sum_{|u| \leq n-1} L(u) \1_{\{L(u|_j) \leq t \text{ for } j<n\}} ([Z_1]_u-1).
\end{equation*}
Further, the martingale $Z_n^{(t)}-1$ converges a.\,s.\ and in $L^\gamma$ to some finite limit $Z^{(t)}-1$
which equals $Z-1$ on $\{\max_{v \in \Gen} |L(v)| \leq t\}$.
We prove that $Z^{(t)}-1$ is also the limit in $L^\gamma$ of the martingale $(M_n^{(t)})_{n \geq 0}$.
Indeed, we write
\begin{align*}
&	\E[|Z^{(t)}-1-M_n^{(t)}|^\gamma]	\\
&=	\E\bigg[\bigg|\sum_{k=0}^\infty \sum_{|u|=k} \!\! L(u) \1_{\{L(u|_j) \leq t \text{\,for\,} j \leq k\}}([Z_1]_u\!-\!1)	\\
&\hphantom{= \E\bigg[\bigg|}~-\sum_{u\in \I_n\cap \Gen} \!\!\! L(u)\1_{\{L(u|_j) \leq t \text{ for } j \leq |u|\}}([Z_1]_u\!-\!1)\bigg|^\gamma\bigg]\\
&=	\E\bigg[\bigg|\sum_{k=0}^\infty \sum_{\substack{|u|=k, \\ u\not \in\I_n}} L(u) \1_{\{L(u|_j) \leq t \text{\,for\,} j \leq k\}} ([Z_1]_u\!-\!1)\bigg|^\gamma\bigg].
\end{align*}
Notice that given $\F_{k}$ the sum $\sum_{|u|=k, u\not \in\I_n} L(u) \1_{\{L(u|_j) \leq t \text{\,for\,} j \leq k\}} ([Z_1]_u\!-\!1)$ is a weighted sum of centered random variables
and can be considered a martingale increment.
Two applications of Lemma \ref{Lem:TV-ineq complex} yield
\begin{align*}
&\E\bigg[\bigg|\sum_{k=0}^\infty \sum_{\substack{|u|=k, \\ u\not \in\I_n}} L(u) \1_{\{L(u|_j) \leq t \text{\,for\,} j \leq k\}} ([Z_1]_u\!-\!1)\bigg|^\gamma\bigg]	\\
&\leq 4 \E\bigg[\sum_{k=0}^\infty \bigg| \sum_{\substack{|u|=k, \\ u\not \in\I_n}} L(u) \1_{\{L(u|_j) \leq t \text{\,for\,} j \leq k\}} ([Z_1]_u\!-\!1)\bigg|^\gamma\bigg]	\\
&\leq 16 \cdot \E[|Z_1-1|^\gamma] \cdot \E\bigg[\sum_{k=0}^\infty \sum_{\substack{|u|=k, \\ u\not \in\I_n}} |L(u)|^\gamma \1_{\{L(u|_j) \leq t \text{\,for\,} j \leq k\}} \bigg].
\end{align*}
The above expectations are finite by \eqref{ass:Lgamma} and the arguments from the proof of Theorem 2.1 in \cite{Kolesko+Meiners:2017}.
By the dominated convergence theorem, we infer that $\lim_{n \to \infty} \E[|Z^{(t)}-1-M_n^{(t)}|^\gamma] = 0$.
Since $\Prob(\max_{v \in \Gen} |L(v)| > t) \leq t^{-\alpha}$, we conclude that $M_n \to Z-1$ a.\,s.
In view of this and \eqref{eq:higher moment},
\begin{align}	\label{eq:8}
\E \bigg[\bigg(\sum_{v\in \Gen} |L(v)|^2|[Z_1]_v-1|^2\bigg)^{\!\!(\alpha+\delta)/2}\bigg]<\infty
\end{align}
by the complex version of Burkholder's inequality.	
On the other hand, from \cite[Theorem 1.5]{Madaule:2016} we have
\begin{equation}	\label{eq:tail of all-time max weight}	\textstyle
\Prob(\max_{u \in \Gen} |L(u)|>t)>ct^{-\alpha}
\end{equation}
for some $c>0$ and all sufficiently large $t$.
Pick $t_0>0$ such that $\Prob(|Z_1-1|>t_0) \geq \frac12$.
Denote by $\mathcal{N}_t$ the set of individuals $u$ that are the first in their ancestral line with the property that $|L(u)| > t$, i.e.,
\begin{equation}
\mathcal{N}_t=\{u\in\I:|L(u)|>t,\text{ and }L(u|_k) \leq t \text{ for all } k<|u|\}.
\end{equation}
Then $\mathcal{N}_t$ is an optional line in the sense of Jagers \cite[Section 4]{Jagers:1989}.
Denote by $\F_{\mathcal{N}_t}$ the $\sigma$-field that contains the information of all reproduction point processes of all individuals
that are neither in $\mathcal{N}_t$ nor a descendent of a member of $\mathcal{N}_t$, see again Jagers \cite[Section 4]{Jagers:1989}
for a precise definition.
Then, by the strong Markov branching property \cite[Theorem 4.14]{Jagers:1989}
(the $\sigma$-field $\F_{\mathcal{N}_t}$ was introduced for a proper application of this result)
and \eqref{eq:tail of all-time max weight}, we infer
\begin{align*}
\Prob\bigg(\sum_{u\in\Gen} |L(u)|^2|[Z_1]_u-1|^2>t_0^2t^2\bigg)
&\geq	\Prob\Big(\sum_{u\in \mathcal{N}_t} |[Z_1]_u-1|^2>t_0^2\Big)\\
&\geq	\Prob(\mathcal{N}_t \not= \emptyset) \cdot \Prob(|Z_1-1|>t_0)
\geq \frac c2t^{-\alpha}
\end{align*}
for all sufficiently large $t$. This contradicts to \eqref{eq:8}, thereby proving that $\nu$ is non-zero.
\end{proof}

\begin{appendix}

\section{Auxiliary results}

In this section, we gather auxiliary facts needed in the proofs of our main results.

\subsection{Inequalities for complex random variables}

Throughout the paper, we need the complex analogues of known inequalities for real-valued random variables.

\subsubsection*{The Topchi\u\i-Vatutin inequality for complex martingales.}
We begin with an extension of the Topchi\u\i-Vatutin inequality
\cite[Theorem 2]{Topchii+Vatutin:1997} to complex-valued martingales.

\begin{lemma}	\label{Lem:TV-ineq complex}
Let $f: [0,\infty)\to [0,\infty)$ be a nondecreasing convex function with $f(0)=0$ such that
$g(x) \defeq f(\sqrt{x})$ is concave on $(0,\infty)$.
Let $(M_n)_{n \in \N_0}$ be a complex-valued martingale with $M_0=0$ a.\,s.\ and
set $D_n \defeq M_n - M_{n-1}$ for $n \in \N$.
If $\E[f(|D_k|)]<\infty$ for $k=1,\ldots,n$, then
\begin{equation}	\label{eq:TV-ineq complex}
\E[f(|M_n|)] \leq 4 \sum_{k=1}^n \E[f(|D_k|)].
\end{equation}
Further, if $f(x)>0$ for some $x>0$ and $\sum_{k=1}^\infty \E[f(|D_k|)] < \infty$, then $M_n \to M_\infty$ a.\,s.\ for some random variable $M_\infty$
and \eqref{eq:TV-ineq complex} holds for $n=\infty$.
\end{lemma}
\begin{proof}
We first observe that
\begin{equation}	\label{eq:bound on f}
f(|z+w|)+f(|z-w|) \leq 2(f(|z|)+f(|w|))	\quad	\text{for all }	z,w \in \C.
\end{equation}
To see this, note that $g$ is subadditive as a concave function with $g(0)=0$, whence
\begin{equation*}
g(x^2+y^2)\leq g(x^2)+g(y^2)=f(x)+f(y)	\quad	\text{for all } x,y \geq 0.
\end{equation*}
Put $x=|u+v|/2$ and $s=|u-v|/2$ for $u,v \in \C$
and observe that
\begin{equation*}
\Big|\frac{u+v}{2}\Big|^2+\Big|\frac{u-v}{2}\Big|^2=\frac{|u|^2+|v|^2}{2}.
\end{equation*}
This together with the concavity of $g$ gives
\begin{equation*}
\frac12 \big(f(|u|)+f(|v|)\big)
= \frac12 \big(g(|u|^2)+g(|v|^2)\big)
\leq g\Big(\frac{|u|^2+|v|^2}{2}\Big)
\leq f\Big(\frac{|u+v|}{2}\Big)+f\Big(\frac{|u-v|}{2}\Big).
\end{equation*}
Multiply this inequality by $2$ and set $u=z+w$ and $v=z-w$ for $z,w \in \C$ to infer \eqref{eq:bound on f}.

The remainder of the proof closely follows the proof of Theorem 2 in \cite{Topchii+Vatutin:1997}.
For $k=1,\ldots,n$ assume that $\E[f(|D_k|)]<\infty$ and denote by $D_k^*$ a random variable such that $D_k$ and $D_k^*$ are i.i.d.\ conditionally given $M_{k-1}$.
Then
\begin{align*}
\E[f(|M_{k-1}+D_{k}-D_{k}^*|)]
&=\E[\E[f(|M_{k-1}+D_{k}-D_{k}^*|) \mid M_{k-1}]]	\\
& =\E[\E[f(|M_{k-1}-(D_{k}-D_{k}^*)|) \mid M_{k-1}]]	\\
&=\E[f[|M_{k-1}-(D_{k}-D_{k}^*)|].
\end{align*}
An appeal to \eqref{eq:bound on f} thus yields
\begin{equation}	\label{eq:separate increment}
\E[f(|M_{k-1}+D_{k}-D_{k}^*|)]
\leq \E[f(|M_{k-1}|)] + \E[f(|D_{k}-D_{k}^*|)].
\end{equation}
Another application of \eqref{eq:bound on f} yields
\begin{equation}\label{inter2}
\E[f(|D_{k}-D_{k}^*|)] \leq 4 \E[f(|D_{k}|)].
\end{equation}
Further,
\begin{align*}
|M_{k-1}+D_{k}|
&= |M_{k-1}+D_{k}-\E[D_{k}^* | M_{k-1}, D_{k}]|	\\
&=|\E[M_{k-1}\!+\!D_{k}\!-\!D_{k}^* \mid M_{k-1}, D_{k}]|	\\
&\leq \E[|M_{k-1}\!+\!D_{k}\!-\!D_{k}^*| \mid M_{k-1}, D_{k}]
\end{align*}
by Jensen's inequality for conditional expectation in $\R^2$,
which is applicable because the function $x\mapsto |x|$ is convex on $\R^2$.
Hence,
\begin{equation}	\label{eq:add 0}
\E[f(|M_{k-1}\!+\!D_{k}|)]
\leq \E[f(\E[|M_{k-1}\!+\!D_{k}\!-\!D_{k}^*||M_{k-1}, D_{k}])]
\leq \E[f(|M_{k-1}\!+\!D_{k}\!-\!D_{k}^*|)]
\end{equation}
by the monotonicity of $f$ and Jensen's inequality.
Combining \eqref{eq:separate increment}, \eqref{inter2} and \eqref{eq:add 0}, we arrive at
\begin{align*}
\E[f(|M_{k}|)]
\leq \E f(|M_{k-1}|)+4\E f(|D_{k}|).
\end{align*}
The claimed inequality follows recursively.

Finally, suppose that $f(x)>0$ for some $x>0$ and that $\sum_{k=1} \E[f(|D_k|)] < \infty$.
Then $\sup_{n \in \N} \E[f(|M_n|)] \leq \sum_{k=1} \E[f(|D_k|)] < \infty$.
Since $f(0)=0$ and $f$ is convex, we conclude that $f$ grows at least linearly fast and, thus,
$\sup_{n \in \N} \E[|M_n|] < \infty$.
By the martingale convergence theorem, $M_n \to M_\infty$ a.\,s.\ for some random variable $M_\infty$.
Further, $f$ is continuous as it is convex and non-decreasing. Therefore, Fatou's lemma implies
\begin{equation*}
\E[f(|M_\infty|)] = \E[\liminf_{n \to \infty} f(|M_\infty|)] \leq \liminf_{n \to \infty} \E[f(|M_n|)] \leq \sum_{k=1}^\infty \E[f(|D_k|)] < \infty.
\end{equation*}
\end{proof}

\subsubsection*{Tail bounds for sums of weighted i.i.d.\ complex random variables.}
Actually, we shall need the following corollary of Lemma \ref{Lem:TV-ineq complex},
which is closely related to Lemma 2.1 in \cite{Biggins+Kyprianou:1997} (see also formulae (2.3) and (2.10) in \cite{Kurtz:1972})
but deals with complex-valued rather than real-valued random variables.

\begin{corollary}	\label{Cor:tail bound sum of weighted iids}
Let $c_1,\ldots, c_n$ be complex numbers satisfying $\sum_{k=1}^n |c_k|=1$.
Further, let $Y_1,\ldots, Y_n$ be independent copies of a complex-valued random variable $Y$
with $\E[Y]=0$ and $\E[|Y|]<\infty$.
Then, for $\varepsilon \in (0,1)$ and with $c\defeq \max_{k=1,\ldots,n} |c_k|$,
\begin{equation*}
\Prob\bigg(\bigg|\sum_{k=1}^n c_kY_k\bigg|>\varepsilon\bigg)
\leq \frac{8}{\varepsilon^2}\bigg(\int_0^{1/c}cx\Prob(|Y|>x) \, \dx + \int_{1/c}^\infty \Prob(|Y|>x) \, \dx \bigg).
\end{equation*}
\end{corollary}
\begin{proof}
We use Lemma \ref{Lem:TV-ineq complex} with $f(x)=x^2\1_{[0,1]}(x)+(2x-1)\1_{(1,\infty)}(x)$ for $x \geq 0$.
Clearly, $f$ is convex such that $g$ defined by
$g(x)=f(\sqrt{x})=x\1_{[0,1]}(x)+(2x^{1/2}-1)\1_{(1,\infty)}(x)$ for $x \geq 0$ is concave.
Furthermore, $f$ is differentiable on $[0,\infty)$ with nondecreasing and continuous derivative
$f'(x)=2x\1_{[0,1]}(x)+2\1_{(1,\infty)}(x)$ for $x \geq 0$.
For $\varepsilon\in (0,1)$, by Markov's inequality,
\begin{align*}
\Prob\bigg(\bigg|\sum_{k=1}^n c_kY_k\bigg| > \varepsilon\bigg)
&\leq \frac{1}{f(\varepsilon)} \E\bigg[f\bigg(\bigg|\sum_{k=1}^n c_kY_k\bigg|\bigg)\bigg]
\leq \frac{4}{\varepsilon^2} \sum_{k=1}^n \E[f(|c_k||Y_k|)]	\\
&=\frac{4}{\varepsilon^2} \sum_{k=1}^n |c_k| \int_0^\infty f^\prime(|c_k|x) \, \Prob(|Y|>x) \, \dx \\
&\leq \frac{4}{\varepsilon^2} \sum_{k=1}^n |c_k| \int_0^\infty f^\prime(c x) \Prob(|Y|>x) \, \dx \\
&=\frac{8}{\varepsilon^2} \bigg(\int_0^{1/c} \!\! cx\Prob(|Y|>x)\, \dx+\int_{1/c}^\infty\Prob(|Y|>x) \, \dx \bigg),
\end{align*}
where we have used Lemma \ref{Lem:TV-ineq complex} and $f(\varepsilon)=\varepsilon^2$ for $\varepsilon\in (0,1)$
for the second inequality, and monotonicity of $f'$ for the third.
Integration by parts gives the first equality.
\end{proof}

\subsection{The minimal position}

In this section, we collect some known results concerning the minimal position in a branching random walk in
what is called the boundary case, see \cite{Biggins+Kyprianou:2005}.

\begin{proposition}	\label{Prop:minimal position}
Let $((V(u))_{|u|=n})_{n \in \N_0}$ be a branching random walk such that the positions in the first generation $V(u)$, $|u|=1$
satisfy the assumptions \eqref{eq:normalized_V}, \eqref{eq:sigma^2<infty Aidekon+Shi} and \eqref{eq:Aidekon+Shi condition}.
Then the sequence of distributions
of $n^{3/2} \sup_{|u|=n} e^{-V(u)}$, $n \in \N$ is tight. In particular,
\begin{equation}	\label{eq:minimal position exponentiated}
n^{1/2} \sup_{|u|=n} e^{-V(u)} \Probto 0	\quad	\text{as } n \to \infty.
\end{equation}
\end{proposition}
Notice that, under some extra non-lattice assumption, \cite[Theorem 1.1]{Aidekon:2013} gives the stronger statement
\begin{equation}	\label{eq:Aidekon's convergence of the minimum in BRW}
\lim_{n \to \infty} \Prob\Big(\min_{|u|=n} V(u) - \frac32 \log n \geq x\Big) = \E[e^{- C^* e^x D_\infty}]
\end{equation}
for all $x \in \R$ where $C^*$ is a positive constant and, as before, $D_\infty$ is the limit of the derivative martingale defined in \eqref{eq:derivative martingale}.
For our proposition, we do not need the full strength of \eqref{eq:Aidekon's convergence of the minimum in BRW},
which allows us to work without the lattice assumption.
\begin{proof}[Proof of Proposition \ref{Prop:minimal position}]
We need to estimate $\min_{|v|=n} V_n(u)$ from below.
Recall that, for $u \in \Gen$, we write $\underline V(u) \defeq \min_{k=0,\ldots,|u|} V(u|_k)$.
With this notation,
we have
\begin{align*}
\Prob\Big[\min_{|v|=n}V_n(u) <-x\Big]
&\leq \Prob\Big(\min_{u\in\Gen}V(u) \leq -x\Big) + \Prob\bigg(\min_{\substack{|u|=n,\\	\underline V(u) \geq - x}}V_n(u)<-x\bigg)    \\
&\leq C(1+x)e^{-x}
\end{align*}
by inequality (4.12) from \cite{Madaule:2017} and \cite[Corollary 3.4]{Aidekon:2013}.
The latter does not require a non-lattice assumption.
\end{proof}

\subsection{Asymptotic cancellation}

Let $L_u = (L_u(v))_{v \in \N}$, $u \in \V$ denote a family of i.i.d.\ copies
of a sequence $(L(v))_{v \in \N} = (L_\varnothing(v))_{v \in \N}$ of complex-valued random variables
satisfying
\begin{equation}\label{eq:finiteness}
\#\{v \in \N: L(v) \not = 0\} < \infty\quad\text{ a.\,s.}
\end{equation}
Define $L(\varnothing) \defeq 1$ and, recursively,
\begin{equation*}
L(uv)	\defeq	L(u) \cdot L_u(v)
\end{equation*}
for $u \in \V$ and $v \in \N$.
Further, we let
\begin{equation*}
Z_n	\defeq	\sum_{|u|=n} L(u)	\quad	\text{and}	\quad	W_n	\defeq	\sum_{|u|=n} |L(u)|
\end{equation*}
where summation over $|u|=n$ here means summation over all $u \in \N^n$ with $L(u) \not = 0$.
Finally, we set $\tilde{W}_1 \defeq \sum_{|v|=1} |L(v)| \log_-(|L(v)|)$.
We extend the shift-operator notation introduced in Section \ref{subsec:background} to the present context,
so if $X = \Psi((L_v)_{v \in \V})$ is a function of the whole family $(L_v)_{v \in \V}$ and $u \in \V$,
then $[X]_u \defeq \Psi((L_{uv})_{v \in \V})$.

\begin{lemma}	\label{Lem:cancellation}
Assume that $\E[W_1] = 1$ and that $a \defeq \E[Z_1] \in \C$ satisfies $|a|<1$.

\noindent
(a) If $\E[\sum_{|v|=1} |L(v)|^p] < 1$ and $\E[W_1^p] < \infty$ for some $p > 1$,
then $Z_n \to 0$ a.\,s. and in $L^{p\wedge 2}$.

\noindent
(b) Suppose that one of the following two conditions holds.
\begin{enumerate}[(i)]
	\item $W_n \to W$ in $L^1$;
	\item $\E[\sum_{|v|=1}|L(v)|\log(|L(v)|)]=0$, $\E [W_1\log_+^2(W_1)]<\infty$, $\E [\tilde{W}_1\log_+(\tilde{W}_1)]<\infty$,
		\begin{align*}
		\E\bigg[\sum_{|v|=1}|L(v)|\log^2(|L(v)|) \bigg]<\infty.
		\end{align*}
\end{enumerate}
Then $Z_n \to 0$ in probability if (i) holds and $n^{1/2} Z_n \to 0$ in probability if (ii) holds.
\end{lemma}

\begin{proof}
(a)
We can assume without loss of generality that $p \in (1,2]$. According to \cite[Corollary 5]{Iksanov:2004} or \cite[Theorem 2.1]{Liu:2000}
the martingale $W_n$ converges a.\,s.\ and in $L^p$ to some limit $W$.
Let
\begin{equation*}	\textstyle
q \defeq \max\{|a|^p, \E[\sum_{|v|=1} |L(v)|^p]\} < 1.
\end{equation*}
For $k=\lfloor n/2\rfloor$, we have
\begin{align*}
\E[|Z_n|^p]
&\leq 2^{p-1} \E[|Z_{n}-a^kZ_{n-k}|^p]+2^{p-1}\E[|a^kZ_{n-k}|^p]\\
&\leq 2^{p-1} \E\bigg[\E\bigg[\bigg|\sum_{|v|=n-k} L(v)([Z_k]_v-a^k)\bigg|^p \, \Big| \, \F_{n-k}\bigg]\bigg] + 2^{p-1} |a|^{kp}\E[W_{n-k}^p]	\\
&\leq 2^{p+1} \E\bigg[\sum_{|v|=n-k} |L(v)|^p \E[|Z_k-a^k|^p] \bigg]+ 2^{p-1} q^{k}\E[W^p]	\\
&\leq 2^{p+1} \E\bigg[\sum_{|v|=n-k} |L(v)|^p\bigg] 2^{p-1} (\E[|Z_k|^p] +|a|^{kp}) + 2^{p-1} q^{k}\E[W^p]	\\
&\leq \big(2^{p+1} (\E[W^p]+1)+ \E[W^p]\big) 2^{p-1} q^k
\end{align*}
where we have repeatedly used that $|z+w|^p \leq 2^{p-1}(|z|^p+|w|^p)$ and Lemma \ref{Lem:TV-ineq complex}
for the third inequality.
The bound decays exponentially as $n \to \infty$ giving
$Z_n \to 0$ in $L^p$ and also $Z_n \to 0$ a.\,s.\ by virtue of the Borel-Cantelli lemma and Markov's inequality.

\noindent
(b)
Let $\Surv \defeq \{W_n > 0 \text{ for all }n \geq 0\}$ denote the survival set of the system.
It is clear that the claimed convergence holds on the set of extinction $\Surv^c$.
Therefore, in what follows we work under $\Prob^*(\cdot) \defeq \Prob(\cdot | \Surv)$.

\noindent
We first assume that (i) holds. Then
\begin{equation}	\label{eq:Biggins martingale convergence W_n}
W_n \to W		\quad	\Prob^*\text{-a.\,s.}
\end{equation}
Eq.\ \eqref{eq:Biggins martingale convergence W_n} in combination with \eqref{eq:Biggins:1998} gives
\begin{equation}	\label{eq:sup/sum->0 a.s.}
\frac{\sup_{|v|=n}|L(v)|}{W_n} \to 0	\quad	\Prob^*\text{-a.\,s.}
\end{equation}
If, on the other hand, assumption (ii) is satisfied,
then \cite[Theorem 1.1]{Aidekon+Shi:2014} gives
\begin{equation}	\label{eq:SH-norming by Aidekon+Shi}
\sqrt n W_n \to  W^*	\quad \text{in $\Prob^*$-probability}
\end{equation}
for some random variable $W^*$ satisfying $\Prob^*(W^* > 0) = 1$.
As before, we need control over $\max_{|u|=n}|L(u)|$.
A combination of
\eqref{eq:SH-norming by Aidekon+Shi} and Proposition \ref{Prop:minimal position} gives
\begin{equation}	\label{eq:sup/sum->0 in prob}
\frac{\sup_{|v|=n}|L(v)|}{\sum_{|v|=n}|L(v)|}=\frac{n^{3/2}\sup_{|v|=n}|L(v)|}{n^{1/2} W_n} \frac1n	\stackrel{\Prob^*}{\to}	 0
\quad	\text{as } n \to \infty.
\end{equation}
From now on, we treat both cases, (i) and (ii), simultaneously.
In view of \eqref{eq:Biggins martingale convergence W_n} and \eqref{eq:SH-norming by Aidekon+Shi}, it remains to prove that
\begin{equation}	\label{eq:remains to show}
\lim_{n \to \infty} \frac{{Z}_n}{W_n} = 0	\quad	\text{in $\Prob^*$-probability.}
\end{equation}
The last relation follows if we can show that, for any fixed positive integer $k< n$,
\begin{equation}	\label{eq:Z_n-a^kZ_n-k/W_n}
\lim_{n\to\infty}\frac{{Z}_n-a^k {Z}_{n-k}}{W_n}=0	\quad	\text{in $\Prob^*$-probability}
\end{equation}
and that, for all $\varepsilon \in (0,1)$,
\begin{equation}	\label{eq:WLLN Z_n-a^kZ_n-k}
\lim_{k\to\infty} \limsup_{n\to\infty}\Prob^*\bigg(\bigg|\frac{a^k {Z}_{n-k}}{W_n}\bigg|>\varepsilon\bigg) = 0.
\end{equation}
Since, for any $k \in \N$, we have
\begin{equation}	\label{eq:W_n-k/W_n}
\lim_{n \to \infty} \frac{W_{n-k}}{W_n} = 1	\quad	\text{in $\Prob^*$-probability},
\end{equation}
for all $k$ such that $|a|^k < \varepsilon/2$, we have
\begin{equation*}
\Prob^*\bigg(\bigg|\frac{a^k {Z}_{n-k}}{W_n}\bigg|>\varepsilon\bigg)
\leq \Prob^*\bigg(\bigg|\frac{a^k {W}_{n-k}}{W_n}\bigg|>\varepsilon\bigg)	\to	0	\quad	\text{as }	n \to \infty.
\end{equation*}
Hence, \eqref{eq:WLLN Z_n-a^kZ_n-k} holds and it remains to check \eqref{eq:Z_n-a^kZ_n-k/W_n}.
In view of \eqref{eq:W_n-k/W_n}, relation \eqref{eq:Z_n-a^kZ_n-k/W_n} is equivalent to
\begin{equation}	\label{eq:Z_n-a^kZ_n-k/W_n-k}
\lim_{n\to\infty}\frac{{Z}_n-a^k {Z}_{n-k}}{W_{n-k}}=0	\quad \text{in $\Prob^*$-probability.}
\end{equation}
Setting $\Surv_j \defeq \{W_j > 0\}$ for $j \in \N_0$, we have $\Surv_j \downarrow \Surv$ as $j \to \infty$.
Thus, for $\varepsilon>0$, we infer
\begin{equation*}
\Prob\bigg(\bigg|\frac{{Z}_n-a^k{Z}_{n-k}}{W_{n-k}}\bigg|>\varepsilon, \Surv\bigg)
\leq \E \bigg[\1_{\Surv_{n-k}}\Prob\bigg(\bigg|\frac{{Z}_n-a^k{Z}_{n-k}}{W_{n-k}}\bigg|>\varepsilon\bigg|\mathcal{F}_{n-k}\bigg)\bigg],
\end{equation*}
so that it suffices to show that the right-hand side converges to zero.
To this end, we work on $\Surv_{n-k}$ without further notice.
We use the representation
\begin{equation*}
\frac{{Z}_n-a^k {Z}_{n-k}}{W_{n-k}} = \sum_{|v|=n-k}\frac{L(v)}{W_{n-k}}([{Z}_k]_v-a^k).
\end{equation*}
Given $\mathcal{F}_{n-k}$, the right-hand side is a weighted sum of
i.i.d.\ centered complex-valued random variables which satisfies
the assumptions of Corollary \ref{Cor:tail bound sum of weighted iids}
with $c_v=L(v)/W_{n-k}$, $|v|=n-k$ and $Y={Z}_k-a^k$.
Note that $\#\{c_v: c_v\neq 0, |v|=n-k\}<\infty$ a.\,s.\ in view of \eqref{eq:finiteness},
that $\sum_{|v|=n-k}|c_v|=1$ a.\,s.\ and that $\E[|{Z}_k-a^k|] \leq \E[|{Z}_k|]+|a|^k\leq \E[W_k]+|a|^k \leq 2$.
With
\begin{equation*}
c(n-k) \defeq \frac{\sup_{|v|=n-k}|L(v)|} {\sum_{|v|=n-k}|L(v)|}
\end{equation*}
an application of Corollary \ref{Cor:tail bound sum of weighted iids} yields
\begin{align*}
\Prob\bigg(\bigg|\frac{{Z}_n-a^k {Z}_{n-k}}{W_{n-k}}\bigg|>\varepsilon\big|\mathcal{F}_{n-k}\bigg)
&\leq
\frac{8}{\varepsilon^2}\bigg(c(n-k) \int_0^{1/c(n-k)} \!\!\!\!\!\! x \, \Prob(|{Z}_k-a^k|>x) \, \dx	\\
&\hphantom{\leq \frac{8}{\varepsilon^2}\bigg(}\,
+\int_{1/c(n-k)}^\infty \Prob(|{Z}_k-a^k|>x) \, \dx \bigg).
\end{align*}
We claim that the right-hand side converges to zero in probability
as $n\to\infty$. Indeed, according to \eqref{eq:sup/sum->0 a.s.} and \eqref{eq:sup/sum->0 in prob}, respectively,
$c(n-k) \to 0$ in $\Prob^*$-probability.
It remains to use the following simple fact.
If $h$ is a measurable function satisfying $\lim_{y \to 0} h(y)=0$ and if $\lim_{n\to\infty} \tau_n=0$
in probability, then $\lim_{n\to\infty} h(\tau_n) = 0$ in probability.
For instance, apply this to $h(y)=y \int_0^{1/y} x \Prob(|{Z}_k-a^k|>x) \, \dx$ and $\tau_n=c(n-k)$.
It follows from Markov's inequality and $\E[|{Z}_k-a^k|]<\infty$ that $\lim_{x \to \infty} x \Prob(|{Z}_k-a^k|>x)=0$,
which in turn implies $h(y) \to 0$ as $y \to 0$.
The proof of \eqref{eq:Z_n-a^kZ_n-k/W_n-k}, and hence of \eqref{eq:Z_n-a^kZ_n-k/W_n}, is complete.
\end{proof}
\end{appendix}

\subsubsection*{Acknowledgements.}
The research of K.\,K.~ and M.\,M.\ was supported by DFG Grant ME 3625/3-1.
K.\,K.~was further supported by the National Science Center, Poland (Sonata Bis, grant number DEC-2014/14/E/ST1/00588).
A part of this work was done while A.\,I.\ was visiting Innsbruck in August 2017.
He gratefully acknowledges hospitality and the financial support again by DFG Grant ME 3625/3-1.

\bibliographystyle{plain}
\bibliography{BRW}

\end{document}